\documentclass{psp}
\usepackage[utf8]{inputenc}
\usepackage[T1]{fontenc}

\usepackage{amsmath}
\usepackage{amssymb}
\usepackage{amscd}
\usepackage{amsfonts}
\usepackage{amsbsy}
\usepackage{graphicx}

\newcommand{\R}{\ensuremath{\mathbb{R}}}
\newcommand{\C}{\ensuremath{\mathbb{C}}}
\newcommand{\N}{\ensuremath{\mathbb{N}}}

\newtheorem {theorem} {Theorem} 
\newtheorem {proposition} [theorem] {Proposition}
\newtheorem {corollary} {Corollary}

\newtheorem {definition} {Definition}
\newtheorem {remark} {Remark}
\newtheorem {example} {Example}

\begin{document}

\title[Hilbert's 16th problem  and Nash space of arcs]
{Hilbert's 16th problem on a period annulus\\ and Nash space of arcs}

\author[J.-P. Fran\c coise, L. Gavrilov and D. Xiao]{JEAN-PIERRE FRAN\c COISE\\
Universit\'e P.-M. Curie, Paris 6, Laboratoire Jacques--Louis Lions \addressbreak
UMR 7598 CNRS, 4 Place Jussieu, 75252, Paris, France \addressbreak and School of
Mathematical Sciences, 
\addressbreak
Shanghai Jiao Tong University, Shanghai,
200240, PR China 
\addressbreak
  e-mail\textup{: \texttt{Jean-Pierre.Francoise@upmc.fr}}  
  \nextauthor
LUBOMIR GAVRILOV\\
Institut de Math\'ematiques de Toulouse; UMR5219 \addressbreak
Universit\'e de Toulouse; CNRS \addressbreak
UPS IMT, F-31062 Toulouse Cedex 9, France
\addressbreak
  e-mail\textup{: \texttt{lubomir.gavrilov@math.univ-toulouse.fr}}
  \and\
DONGMEI XIAO\\
School of Mathematical Sciences \addressbreak
Shanghai Jiao Tong University, Shanghai,
200240, PR China  \addressbreak
  e-mail\textup{: \texttt{xiaodm@sjtu.edu.cn}}}



\maketitle
\begin{abstract}
This article introduces an algebro-geometric setting for the space of bifurcation functions involved in the local Hilbert's 16th problem on a period annulus. Each possible bifurcation function is in one-to-one correspondence with a point in the exceptional divisor $E$ of the canonical blow-up  $B_I{\C}^n$ of the Bautin ideal $I$. In this setting, the notion of essential perturbation, first proposed by Iliev, is defined via irreducible components of the Nash space of arcs $ Arc(B_I\C^n,E)$. The example of planar quadratic vector fields in the Kapteyn normal form is further discussed.

\end{abstract}
\newpage
\tableofcontents

\section{Introduction}

In full generality, this article deals with bifurcation theory of polynomial planar vector fields $X_\lambda$ depending of a set of parameters $\lambda = (\lambda_1,...,\lambda_n)\in \Lambda$. We assume that the ``center set" of vector fields $X_\lambda$ having a center
is an affine algebraic variety defined by an ideal
 in the ring of polynomials in $\lambda$ (the so called  Bautin ideal associated to $X_\lambda$). 

In the history of bifurcation theory, many reductions of bifurcation problems have been made  ``by hand" and sometimes without the need of full justifications. 
It turns out  that in this context the Nash space of arcs/jets often provides the right setting.

We still remain to a very elementary level for specialists of algebraic geometry, although it seems interesting to propose here a first application of Nash space of arcs to complex/real foliations and bifurcation theory.

\subsection{Hilbert's 16th problem on a period annulus}
\label{openannulus}

 An open  period annulus $\Pi$ of polynomial planar vector field $X_{\lambda_0 }$ is a union of periodic orbits of $X_{\lambda_0 }$,
which is homeomorphic to the standard annulus $S^1\times (0,1)$, the image of each periodic orbit being a circle. We consider an {\it unfolding} $X_{{\lambda }}$ of $X_{\lambda_0 }$ which depends on finitely many parameters $\{{\lambda }=(\lambda_1,...\lambda_n)\in \Lambda\subset \R^n\}$, where $\Lambda$ belongs to a small ball centered at $\lambda_0$ in the parameter space $\R^n$. 
The (infinitesimal) 16th Hilbert problem on the period annulus $\Pi$ is to find an uniform bound in ${\lambda }$, on the number of limit cycles of $X_{{\lambda }}$, which tend to $\Pi$ as $\lambda$ tends to $\lambda_0$. The precise meaning of this is encoded in the notion of cyclicity $Cycl(\Pi,X_{\lambda_0 },X_{{\lambda }})$ , which we define below, see \ref{Cyclicity}. However, except in some particular cases it is not even known whether such a bound exists, e.g. \cite{Gavrilov3,gano08,bny10}.

The reader can think, as possible examples, to perturbation of a quadratic center by a quadratic planar vector field, 
which we revisit in section \ref{quadratic}.

Let $\Sigma$ be 
an open transversal cross-section to $X_{\lambda_0 }$ on the open set $\Pi$,  $\bar{\Sigma} \subset \Pi$. We further assume that $X_{\lambda_0 }$ is also transverse at the boundary points of $\overline{\Sigma}$.
For $\lambda$ close to $\lambda_0$, $X_{{\lambda }}$ remains transverse to $\Sigma$ and there is an analytic first return map 
${\mathcal P}_{{\lambda } }: \Sigma \times \Lambda \mapsto \Sigma', \overline{\Sigma}\subset \Sigma'$, with $\overline{\Sigma'}\subset\Pi$. The limit cycles of $X_{{\lambda }}$ are in one-to-one correspondence with the fixed points of  ${\mathcal P}_{{\lambda }}$ and hence with the zeros of the displacement function
 $$h \mapsto F(h,{\lambda } )= {\mathcal P}_{{\lambda }} (h)- h$$
 in its domain of definition. The coefficients $F_k({\lambda }), (k>1)$ of the analytic convergent series (in $h$, coordinate on $\Sigma$ so that $0\in\Sigma,h(0)=h_0$):
\begin{eqnarray}
\label{displacement}
\begin{aligned}
F(h,{\lambda })=\Sigma_{k=1}^{+\infty} F_k({\lambda })(h-h_0)^k,
\end{aligned}
\end{eqnarray}
are  analytic also in ${\lambda }$ in a neighbourhood of $\lambda_0$.

The infinitesimal 16th Hilbert problem on the period annulus $\Pi$ asks, alternatively,  to find a bound  on the number of fixed points of the first return map 
$h \mapsto {\mathcal P}_{{\lambda }}(h)$, which is uniform in  ${\lambda }$. In this context  ${\lambda }$ will belong to some sufficiently small neighbourhood of a given $\lambda_0$, which belongs to the center set.

 The problem which we consider should not be confused with the study of the displacement function on the closed period annulus $\bar{\Pi}$. In particular the study of $F(h,{\lambda })$  in a neighbourhood of a polycycle, or a slow-fast manifold is beyond the scope of the paper.

\subsection{Cyclicity}
\label{Cyclicity}
We follow \cite{rous98, gano08, Gavrilov}.  As in section \ref{openannulus}, consider 
 a family $\{  X_{{\lambda }} \}_{ {\lambda } \in \Lambda}$ of polynomial planar real vector fields which depend analytically on finitely many parameters
 $$\{{\lambda }=(\lambda_1,...\lambda_n)\in \Lambda\subset (\R^n,0)\}  $$
and let    $\Pi \subset \R^2$  be an open period annulus of $X_{\lambda_0 }$.
For an arbitrary compact set $K\subset \Pi$ we define its cyclicity $Cycl(K,X_{\lambda_0 },X_{{\lambda }})$  as the maximal number of limit cycles of the vector field $X_{{\lambda }}$, which tend to $K$ as ${{\lambda }}$ tends to $\lambda_0$. This allows to define the cyclicity of the open period annulus $\Pi$ as
\begin{align}
Cycl(\Pi,X_{\lambda_0 },X_{{\lambda }}) &= \underset{K\subset \Pi}{\sup} \{ Cycl(K,X_{\lambda_0 },X_{{\lambda }}) :  K \mbox{ is a compact} \}
\end{align}
\cite[Definition 3]{Gavrilov}.

The conjectural finiteness of the cyclicity of period annuli (closed or open) of polynomial vector fields is a largely open problem, inspired by the second part of the 16th Hilbert problem, see \cite[Roussarie, section 2.2]{rous98}. A localized version of Theorem 1 (around a single Hamiltonian cycle) was proved in \cite[Roussarie]{Rous}.Through this paper we assume that

\begin{align}
Cycl(\Pi,X_{\lambda_0 },X_{{\lambda }})& < \infty .
\label{finite}
\end{align}

\subsection{One-parameter unfoldings which maximize the cyclicity}
\begin{figure}
\includegraphics[width=10cm]{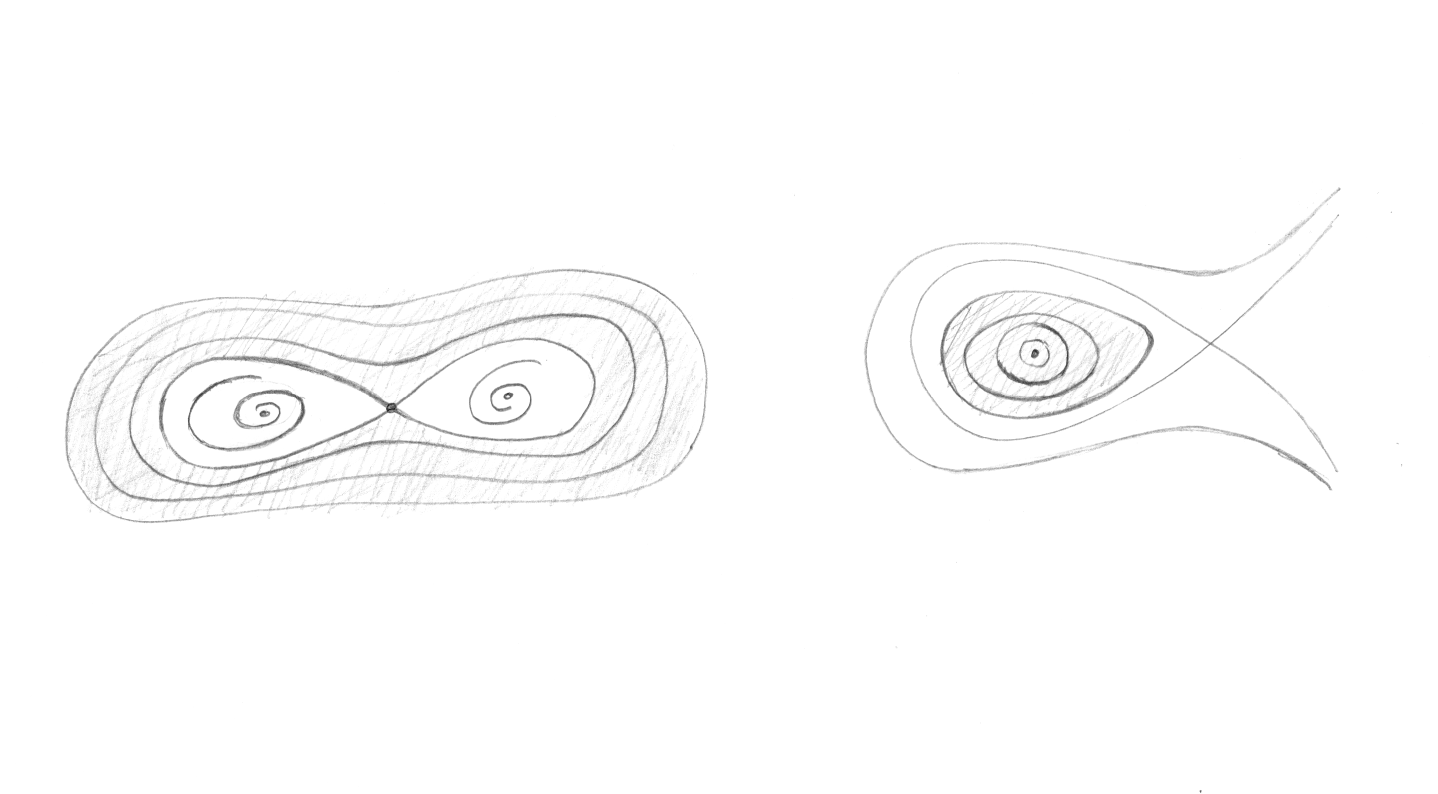}
\caption{Period annuli}
\label{annuli}
\end{figure}
Given an analytic family of vector fields  $\{  X_{{\lambda }} \}_{ {\lambda } \in \Lambda}$  we may consider germ of analytic arcs
\begin{equation}
\label{arce}
\varepsilon  \mapsto  {\lambda } (\varepsilon),{\lambda }  (0)=\lambda_0
\end{equation}
and the induced one-parameter families of vector fields $\{  X_{{\lambda }(\varepsilon)} \}$. Obviously we have
$$
Cycl(\Pi,X_{\lambda_0 },X_{{\lambda } (\varepsilon)}) \leq Cycl(\Pi,X_{\lambda_0 },X_{{\lambda }}) .
$$
At a first sight, it is restrictive to study only one-parameter deformations (arcs in the parameter space).
The following result shows that if we consider \emph{families } of one-parameter deformations (families of arcs  in the parameter space), then 
the two approaches give the same answer
\begin{theorem}[\cite{Gavrilov,Rous}]
\label{gav}
Under the finiteness condition (\ref{finite}),
there exists an analytic arc  (\ref{arce})
such that the equality holds
$$Cycl(\Pi,X_{\lambda_0 },X_{{\lambda } (\varepsilon)}) =  Cycl(\Pi,X_{\lambda_0 },X_{{\lambda }}) .$$
\end{theorem}
The proof relies on two ingredients, the principalization of the
ideal of the center set by blowup (cf. \cite{Hartshorne}, ch. II,
7.13 and see subsection 4.1) and a global version of the Weierstrass
preparation theorem, applied to the displacement map $F$. This shows
that the complement to the bifurcation set of limit cycles (isolated
zeros) is a sub-analytic subset of $\Lambda$. Applying the ``curve
selection lemma" we obtain the analytic arc  in question. \vskip 1pt

The main question addressed in our article is about how to construct all one-parameter deformations, or arcs in the parameter space. 
As far, as we are interested in cyclicity, it is clear that most of the one-parameter  deformations are redundant. To avoid redundancy, we shall consider only ``essential" deformations , and moreover we shall organise them in algebraic families of one-parameter deformations. The key observation is that to parametrize these families of arcs, we should use the associated bifurcation functions.

\subsection{The bifurcation function  of a one-parameter unfolding}
Consider an one-parameter analytic unfolding $X_\varepsilon$ of the vector field with a center $X_{ 0 }$, that is to say a perturbation of $X_0$. 
The displacement function associated to $X_\varepsilon$  can be developed in a power series in $\varepsilon$
\begin{eqnarray}
\label{displacement1}
F(h,\varepsilon)=\Sigma_{i = k}^\infty M_i(h) \varepsilon^i , M_k \neq 0 .
\end{eqnarray}
The leading term $M_k$  is called 
 the bifurcation function, or $k$-th order Melnikov function, associated to the unfolding $X_\varepsilon$ \cite{Iliev,F,Zoladek1}.
 Let $\Sigma$ be now a global cross-section of the period annulus $\Pi$ of $X_0$. The displacement function $F(h,\varepsilon)$ is defined on an open, relatively compact subset of $\Sigma$, depending on $\varepsilon$.
 An important feature of $M_k$ is that, in contrast to $F(.,\varepsilon)$, it can be defined on the whole open interval $\Sigma$ and it is analytic on it \cite{Gavrilov2}. 
 
Possible bifurcations of limit cycles from the ovals of $X_0$ correspond to zeros of the displacement function, and hence to zeros of  the bifurcation function $ M_k$ on $\Sigma$. 
 Thus, if $M_k$ is associated to an one-parameter unfolding, maximizing cyclicity of $\Pi$ with respect to $X_{\lambda }$ , then the zeros of $M_k$ on $\Sigma$ provide an upper bound to this cyclicity  $Cycl(\Pi,X_{\lambda_0 },X_{{\lambda }})$.  \emph{To solve the  infinitesimal 16th Hilbert problem on the open period annulus $\Pi$, amounts to study zeros of all bifurcation functions associated to arcs (\ref{arce}). 
} 
\subsection{One-parameter perturbations as arcs on singular varieties}
Given a perturbation $X_\varepsilon$ we associate a bifurcation function.
To avoid redundancy, we parameterize perturbations $X_\varepsilon$ by bifurcations functions, and ask for families of perturbations $X_\varepsilon$, which produce all possible bifurcation functions.
 Such remarkable families of perturbations (if they exist!) were called "essential" by Iliev, and studied in detail in the quadratic case \cite{Iliev}. 
Our approach fits into the Nash theory of arcs on singular varieties. A perturbation $X_\varepsilon$ becomes an arc on the  blow up of the Bautin ideal, related to the period annulus $\Pi$. The bifurcation functions are identified to the exceptional divisor of the blow up. The Iliev essential perturbations turn out to be special irreducible components of the associate Nash space of arcs. 

 \subsection{Plan of the paper.}
 The paper has three parts. 
 
  In  section \ref{section2} we describe some algebro-geometric background, needed to study the blow up of an ideal via the Nash theory of arcs.
   
   In section   \ref{section3}  we develop a dictionary between section  \ref{section2} (Nash space of arcs) and the problem, announced in the title of the paper : arcs are identified to one-parameter vector fields $X_\varepsilon$ (perturbations), bifurcations functions are identified to points on an exceptional divisor of blowup. As a byproduct we obtain finiteness results on the order of the bifurcation (or Melnikov) functions, as well a geometric description of the Iliev essential perturbations.
  
In the last section \ref{quadratic} we illustrate our approach on the family of plane quadratic vector fields in the so called Kapteyn normal form.
It tuns out that there are only five irreducible components of the Nash space, corresponding to essential perturbations.
  
\section{Blow-up of an ideal and its space of arcs}
\label{section2}
\subsection{Blow-up of an Ideal}

Let $I=(v_1,\dots,v_N) \subset \mathbb C
[{\lambda }]$, ${\lambda }=(\lambda_1,...\lambda_n)\in  \C^n$, be an ideal with zero set
$$Z(I)=\{{\lambda }\in \C^n : v_1({\lambda })=v_2({\lambda })= \dots = v_N({\lambda })=0\}.
$$
The blowup $B_I{\C}^n \subset  \C^n\times \mathbb P^{N-1}$ of $\C^n$
with center $I$ is the Zarisky closure
 of the graph of the map
\begin{equation}
\label{blowup}
\C^n \setminus Z \to \mathbb P^{N-1}, {\lambda }  \mapsto
[ v_1({\lambda }):\dots : v_N({\lambda })]
\end{equation}
with projection on the first factor $\pi : B_I{\C}^n \to \C^n $ .
Here $ [ v_1({\lambda }):\dots : v_N({\lambda })]$
is the projectivization of the vector $ (
v_1({\lambda }),\dots , v_N({\lambda }))$.

We say that  $\pi$ is the blow up map of  $\C^n$ with center at $I$,
and $E=\pi^{-1}(Z)$ is the exceptional locus. 
Here $Z$ and $E$ are algebraic varieties, which are not necessarily
smooth manifolds. For every $\lambda\in Z$ we denote by $E_\lambda =
\pi^{-1}(\lambda)$ the fibre of $E$ over $\lambda$. The fibre
$E_\lambda\subset  \mathbb P^{N-1}$ is a projective variety. 

The above construction is in fact local, and moreover depends only on the ideal $I$, not on the choice of generators $(v_1,...,v_N)$, see \cite[chapter II, 7]{Hartshorne}.
Therefore, we may replace the ideal $I$ by the ideal sheaf $\mathcal I$ generated by $I$  in the sheaf of rings of convergent power series $\mathcal O_{\C^n}$.
The blowup of the ideal sheaf $\mathcal I$ leads in a neighbourhood of a given point $\lambda \in Z(\mathcal I)$  to a variety   isomorphic to $ B_I{\C}^n$ in a neighbourhood of $\pi^{-1}(\lambda)$.
The main property of the blowup is the fact, that the restriction of the ideal sheaf $\mathcal I$ on it is principal. To be more precise, define the inverse  image ideal sheaf $\tilde{\mathcal I}$ of $\mathcal I$ by
$$
\tilde{ \mathcal I}  =  \pi^{-1} (\mathcal I) \cdot \mathcal O_{B_I{\C}^n} .
$$
which is roughly speaking the pre-image of $\mathcal I$ on the  surface $B_I{\C}^n$. Then
\begin{proposition}
The inverse  image ideal sheaf $\tilde{\mathcal I} $ on the blowup $B_I{\C}^n$ is a principal ideal sheaf.
\end{proposition}
The proof is essentially tautological  \cite[II, Proposition 7.13(a)]{Hartshorne}. 
A hint to the proof is provided by the next two basic examples.
\begin{example}
If $I=(\lambda_1,\lambda_2)$ then the  blowup $B_I{\C}^2$ is a smooth surface   covered by two charts 
$$U_1 = \{(\lambda_1,\lambda_2,\nu) \in \C^3:\lambda_1=\nu \lambda_2 \}, U_2 = \{(\lambda_1,\lambda_2,\mu) \in \C^3:\lambda_2=\mu \lambda_1 \}
$$
identified by the relation $\nu\mu=1$. The inverse image ideal sheaf of the ideal sheaf $\mathcal I$ on the surface $B_I{\C}^2 = U_1\cup U_2$ is principal. More precisely, on the chart
$U_1$ we have $(\lambda_1,\lambda_2) = (\lambda_2)$ (because $\lambda_1=\nu \lambda_2$) and on the chart
$U_2$ we have $(\lambda_1,\lambda_2) = (\lambda_1)$ (because $\lambda_2=\mu \lambda_1$) .
\end{example}
\begin{example}
\label{ex1}
For the ideal   $I=(\lambda_1,\lambda_2^{k+1})$, $k\geq 1$, the blowup $B_I\C^2$ is a  surface, 
covered by two charts
$$U_1 = \{(\lambda_1,\lambda_2,\nu) \in \C^3:\lambda_1=\nu \lambda_2^{k+1} \}, U_2 = \{(\lambda_1,\lambda_2,\mu) \in \C^3: \lambda_2^{k+1} =\mu \lambda_1 \}
$$
identified by the relation $\nu\mu=1$. 
 The inverse image ideal sheaf $ (\lambda_1,\lambda_2^{k+1})$ is principal on $B_I\C^2$. More precisely, on $U_1$ it is generated by $ \lambda_2^{k+1}$ and on $U_2$ it is generated by $\lambda_1$.
The fibre $E=E_0$ is  $\mathbb P^1$ but in contrast to the case $k=0$, the surface  $B_I\C^2 \subset \C^2 \times \mathbb P^1$ is singular at a single point $\lambda_1=\lambda_2=\mu=0$  : we get the singularity of type  $A_k$ 
$$
\{(\lambda_1,\lambda_2,\nu) \in \C^3 : \lambda_1 \mu  = \lambda_2^{k+1} \}
$$
which is the basic example in which the Nash space of arcs is easily computed, see \cite[p.36]{Nash} or \cite[Example 9]{joko16}.

\end{example} 

We resume the analytic counterpart of the above claims as follows. Let $\mathcal I_{\lambda_0}$ be the germ of ideal defined by $I$ at the point $\lambda_0 \in Z$, and $u_1',u_2',\dots , u'_{k'} \in \mathcal I_{\lambda_0}$ be a set of germs of analytic functions, which generate $\mathcal I_{\lambda_0}$ . 
 We may repeat the above construction to the graph of the map
\begin{align*}
U  &\to \mathbb P^{k'-1} \\
{\lambda }  &\mapsto
[ u_1'({\lambda }):u_2'({\lambda }): \dots : u_k'({\lambda })]
\end{align*}
where $U$ is a suitable neighbourhood of $ \lambda_0$,
by taking its closure $X' \subset U\times P^{k'-1}$ in complex topology. 
Similarly, if $u_1'',u_2'',\dots , u_{k''}''$ is another set of generators of $\mathcal I_{\lambda_0}$,  then we may construct the blowup $X'' \subset U\times P^{k''-1}$, provided that $U$ is a suitable neighbourhood of $\lambda_0$.

\begin{proposition}
\label{blowupp}
The blowups $X' $ and $X'' $ are analytic sets
and there is an analytic isomorphism $ f: X' \to X''$ such that
\begin{description}
\item[ (i) ] $f$ commutes with the projection maps $$ \pi' : X' \to U,\; \pi'': X'' \to U, \;  \pi'= \pi'' \circ f$$
\item[(ii)] $f$ induces a linear isomorphism between the fibres $$ (\pi')^{-1}(\lambda_0) \subset \mathbb P^{k'} \mbox{  and  }  (\pi'')^{-1}(\lambda_0) \subset \mathbb P^{k''}$$
\end{description}
\end{proposition}
\begin{proof}
Let
\begin{align}
\label{eq1}
u'=(u_1',u_2',\dots , u_{k'}') , u'' = (u_1'',u_2'',\dots , u_{k''}'')
\end{align}
be two sets of germs of analytic functions at $\lambda_0$ generating  $\mathcal I_{\lambda_0}$. There exist matrices
\begin{align}
\label{eq2}
a' = (a_{ij}') , a'' = (a_{ks}'')
\end{align}
with coefficients in  $\mathcal O_{\lambda_0}$ and such that
\begin{align}
\label{eq3}
u'' =  u'a' , u' = u'' a''
\end{align}
Let $\lambda= \lambda(\varepsilon)$ be an arc centered at $\lambda_0$, and
\begin{align}
\label{eq4}
u'(\lambda(\varepsilon)) = \varepsilon ^{k'} p' (1+ O(\varepsilon)),
u''(\lambda(\varepsilon)) = \varepsilon ^{k''} p'' (1+ O(\varepsilon))
\end{align}
where $p' , p''$ are non-zero vectors. It follows from (\ref{eq3}), (\ref{eq4}) that $k'=k''$ and
$$
p''= a'(\lambda_0) p', p' = a''(\lambda_0) p'' .
$$
\end{proof}

\subsection{The Nash space of arcs}
Suggested references to this section are  \cite{MLJ1, MLJ2, joko16, Nash}.  
Let $X$ be an algebraic variety (possibly singular). 
A formal arc $\alpha$
 is a parameterized formal curve
  \begin{equation}
  \varepsilon \to \alpha(\varepsilon) \in X .
   \end{equation}
   The set of $k$-jets of such arcs is an algebraic variety $X_k$, and there is a canonical projection $X_i \to X_j$ for $i\geq j$. The projective limit 
   $Arc(X) =\varprojlim X_i$ is therefore a proalgebraic variety, called the Nash space of arcs on $X$. 

Let $X_{sing}$ be the singular locus of $X$, or more generally, any algebraic subset of $X$.
A formal arc $\alpha$
  centered at  $X_{sing}$ is a parameterized formal curve
  \begin{equation}
  \label{nash}
  \varepsilon \to \alpha(\varepsilon) \in X,\;\; \alpha(0) \in X_{sing}
 \end{equation}
  which meets $X_{sing}$ at $\varepsilon=0$.
  The space of all such formal arcs is a  proalgebraic variety defined similarly, as a projective limit of $k$-jets of arcs, centered at $X_{sing}$. It is  denoted $Arc(X,X_{sing})$.

In this section we assume that $X=B_I\C^n$, and $E=\pi^{-1}(Z(I))$ is the exceptional locus of the blow-up. 

A general arc $\alpha \in Arc(B_I\C^n,E )$ is not contained in $E$, so it can be described by its projection on the $\lambda$-plane 
$\lambda(\varepsilon)= \pi (\alpha(\varepsilon))$ and vice versa. Of course, the   topology on the space of arcs $\pi (\alpha(\varepsilon))$ is the one, induced by the topology on the space of arcs on $B_I\C^n$.
The arc
$$
\varepsilon \to \lambda(\varepsilon)
$$
is a formal parameterized curve on $\C^n$, which meets the zero locus $Z(I)$ at $\varepsilon = 0$, and is not contained in $Z(I)$. The exceptional locus $E$ will be in general a complicated singular set, which can be studied by further desingularization of $B_I{\C}^n$. 
But is it possible to describe the geometry of $E$ without doing this? It turns out that, as suggested by Nash,  it is enough to study all arcs passing through a point $P\in E$.
\begin{proposition}
\label{exceptional} 
$P\in E $ if and only if there is an
analytic arc  
\begin{equation}\label{arc}
\C,0 \to \C^n: \varepsilon  \mapsto \lambda(\varepsilon)
\end{equation}
not contained in the zero set $Z=Z(I)$, such that $\lambda(0) \in
Z$ and
\begin{equation}
\label{admissible} P= \lim_{\varepsilon \to 0}   [
v_1(\lambda(\varepsilon)):\dots : v_N(\lambda(\varepsilon))] .
\end{equation}

\end{proposition}
	\begin{proof}
 Let $P\in E$ and consider a
resolution of $B_I\C^n$ :
		\begin{equation}
\label{resolution} R_I{\C}^n \stackrel{\tilde{\pi}}{\to }  B_I{\C}^n
\stackrel{\pi}{\to} \C^n
		\end{equation}
By this we mean that  $R_I{\C}^n$ is a smooth variety,  and the
projection $\tilde \pi$  is a bi-rational morphism, which is
bijective over the complement $B_I\C^n \setminus E$. Let
$\tilde{P} \in \tilde{\pi}^{-1}(P)$ be some pre-image of $P$ on $R_I{\C}^n$.
As the latter is smooth, then there exists an arc
 $\tilde{\alpha} : \C,0  \to R_I{\C}^n$   with $\tilde{\alpha}(0)= \tilde{P}$, not contained in the divisor $\tilde{\pi}^{-1}(E)$.
 Then the projection of the arc $\tilde{\alpha}$ on $B_I{\C}^n$ is an analytic arc $\alpha$ which meets $E$ at $P$, and 
 the projection $\pi(\alpha(\varepsilon))= \lambda(\varepsilon)$ is an analytic arc on $\C^n$ with $\lambda(0)= \pi(\alpha(0)) \in Z$.
	\end{proof}

The existence of the limit (\ref{admissible}) is equivalent to the
existence of a natural number $k\geq 1$ such that
\begin{equation}
\label{vk} (v_1(\lambda(\varepsilon)),\dots ,
v_N(\lambda(\varepsilon)))= \varepsilon^k (1+ O(\varepsilon)) p
\end{equation}
where      $p\in \C^N$ is a non-zero vector whose  projectivization
is the point $P$. The construction of $k$ is local, so we could replace the generators $v_1,v_2, \dots, v_N$ by their localizations at $\lambda_0$ in  $\mathcal I_{\lambda_0}$. We can even replace $v_1,v_2, \dots, v_N$ by another set of generators of the localized ideal  $\mathcal I_{\lambda_0}$. It follows from the proof of Proposition \ref{blowupp}
that
\begin{corollary}
\label{orderk}
Given an arc $\alpha \in Arc(B_I{\C}^n, E_{\lambda_0})$ the number $k$ defined in (\ref{vk}) does not depend on the choice of generators of the germ of ideal   $\mathcal I_{\lambda_0}$ defined by $I$ at the point $\lambda_0$.
\end{corollary}
\begin{definition}[Iliev numbers]
\label{def1}
Given an arc $\alpha \in Arc(B_I{\C}^n, E)$, we call $k=k(\alpha)$ defined in   (\ref{vk}), the order of $\alpha$ at the center $P= \alpha(0)$.
Let $P\in E$ be fixed, and let
$k_P$ be the minimal order which an arc $\alpha$ centred at $P$, can have
 $$
 k_P=\min_{\alpha(0)=P}  k(\alpha) .
 $$
For a given fixed $\lambda_*\in Z(I)$  define further
\begin{equation}
\label{k*} k_*= \sup_{\pi(P)=\lambda_*} k_P 
\end{equation}
and  
\begin{equation}
\label{kk} k_{max}= \sup_{P\in E} k_P = \sup_{\lambda_*\in Z} k_*
\end{equation}
\end{definition}

Such numbers are further called Iliev numbers. The next result  says that all these numbers are
finite
\begin{theorem}
\label{finitek}
$$
k_{max} < \infty .
$$
\end{theorem}
\begin{proof}
Let us suppose that (\ref{resolution}) is a strong resolution, in
the sense  that $\tilde{\pi}^{-1}(E)$ is a divisor with simple
normal crossing. The inverse image ideal sheaf $$\tilde{I} = (\pi
\circ \tilde{\pi})^* I$$ is locally principal and locally monomial,
and its zero locus is just $\tilde{\pi}^{-1}(E)$. Thus in a
neighborhood of each point $\tilde{P} \in \tilde{\pi}^{-1}(E)$ we
can find local coordinates $z_i$ and natural numbers $c_i$, such
that the ideal sheaf $\tilde{I}$ is generated by $ \Pi_i z_i^{c_i} .
$ We define the order of vanishing, or order of the locally
principal ideal sheaf $\tilde{I} =(\pi \circ \tilde{\pi})^* I$ at
$\tilde{P}$  to be
$$ord_{\tilde{P}} \tilde{I} =\sum_i c_i .$$
As $\tilde{\pi}^{-1}(E) $ is a subvariety of $R_I\C^n$, then it has
a finite number of irreducible components, locally defined by
$z_i^{c_i}=0$. It follows that the number
$$
\mbox{max-ord}:=\max \{ ord_{\tilde{P}} \tilde{I} : \tilde{P} \in
\tilde{\pi}^{-1} (E) \}
$$
is finite.

Let $P\in E$ and $\tilde{P}$ be a pre-image of $P$ under
$\tilde{\pi}$ as in the proof of Proposition \ref{exceptional}.
Consider an arc $\tilde{\alpha}$ which coincides with a general
straight line through $\tilde{P}$ in local coordinates $z_i$. The
local principality of $\tilde{I}$ implies that
\begin{equation}
\label{vkt} (\tilde{v}_1(\tilde{\alpha}(\varepsilon)),\dots ,
\tilde{v}_N(\tilde{\alpha}(\varepsilon)))= \varepsilon^k (1+
O(\varepsilon)) p
\end{equation}
where $k= \sum_i c_i \leq  \mbox{max-ord}$, $v_i \circ \pi \circ
\tilde{\pi}= \tilde{v}$, and $p$ is a non-zero vector. The
projection of the arc $\tilde{\alpha}$ under $ \pi \circ
\tilde{\pi}$ on $\C^n$ gives an analytic arc $\varepsilon\to \lambda(\varepsilon)$, such that
$$(v_1(\lambda(\varepsilon)),\dots , v_N(\lambda(\varepsilon)))= \varepsilon^k (1+ O(\varepsilon)) p $$
(with the same $k$ as in (\ref{vk})) and $P$ is the
projectivization of $p$. The number $\mbox{max-ord}$ is therefore an
upper bound for the number $k_P$. As $\mbox{max-ord}$ does not
depend on $P$, then the finiteness of the Iliev numbers is proved.
\end{proof}
\begin{definition}
\label{def2}
We define $\mathcal M_k \subset  Arc(B_I\C^n,E)  
$ to be the set of arcs  of order at most $k$, that is to say
\begin{equation*}
 (v_1(\lambda(\varepsilon)),\dots ,
v_N(\lambda(\varepsilon)))= \varepsilon^i (1+ O(\varepsilon)) p_\alpha
\end{equation*}
where $i \leq k$ and $p_\alpha\in \C^N$ is a non-zero vector. 
\end{definition}
We get therefore a filtration $\mathcal M_1 \subset \mathcal M_2 \subset \dots \mathcal M_k \subset \dots$
of $Arc(B_I\C^n,E) $

\begin{proposition}
\label{mk}
The closure of 
 $\mathcal M_k$ is a union of  irreducible components of the arc space $Arc(B_I\C^n,E)$.
\end{proposition}
\begin{proof}
Given an arc  $\alpha$ with projection $\pi(\alpha) : \varepsilon \mapsto \lambda(\varepsilon)$ 
we note that a continuous (in the Nash topology on  $Arc(B_I\C^n,E)$) deformation $s \to \alpha_s$ of $\alpha$ induces a continuous deformation of the projection
$ \varepsilon \mapsto \lambda(\varepsilon)$ and therefore a continuous deformation of
 $\varepsilon \mapsto v(\varepsilon)=(v_1(\lambda(\varepsilon)),v_2(\lambda(\varepsilon)),\dots, v_N(\lambda(\varepsilon)))$. Under such a deformation the order can not increase.
This shows, that for sufficiently small $s$ the arc $\alpha_s$ still belongs to  $\mathcal M_k$, and hence the irreducible component of  $Arc(B_I\C^n,E)$ containing $\alpha_0$ belongs to  $\mathcal M_k$ too.
\end{proof}
The claim of the above Proposition can be reformulated as follows
\begin{proposition}
The order 
$$
Arc(B_I\C^n,E) \to \N : \alpha \mapsto k(\alpha)
$$
is an upper semi-continuous function on $Arc(B_I\C^n,E) $.
\end{proposition}

\begin{example}
\label{ex2}
We revisit the polynomial ideal $I=(\lambda_1,\lambda_2^{k+1})\subset \C[\lambda_1,\lambda_2]$ from Example \ref{ex1}, with the same notations. We have
$$
\mathcal M_1 = \{ \alpha : \lambda_1(\varepsilon)= \varepsilon \lambda_1^{(1)} + \varepsilon^2  \lambda_1^{(2)} + \dots ,  
\lambda_2(\varepsilon)= \varepsilon \lambda_2^{(1)} + \varepsilon^2 \lambda_2^{(1)} + \dots, \lambda_1^{(1)} \neq 0 \}
$$
and $\mathcal M_1$ is an irreducible component of the Nash space $ Arc(B_I\C^n,E)$, freely parameterised by $\lambda_1^{(i)}, \lambda_2^{(j)}, \lambda_1^{(1)}\neq 0$. Similarly, for $i=2$ the algebraic set $\mathcal M_2$ is an union of $\mathcal M_1$ and
$$
\mathcal M_2 \setminus \mathcal M_1 = 
\{ \alpha : \lambda_1(\varepsilon)=   \varepsilon^2 \lambda_1^{(2)} + \dots ,  \lambda_2(\varepsilon)= \varepsilon \lambda_2^{(1)} + \varepsilon^2 \lambda_2^{(2)} + \dots, \lambda_1^{(2)} \neq 0 \} .
$$
We note that $\mathcal M_2 \setminus \mathcal M_1$ is not in the closure of $\mathcal M_1$. 
Indeed, our arcs live on the blown-up surface $  B_I\C^n $, which in affine coordinates is 
$\lambda_1\mu=\lambda_2^{k+1}$. For $\alpha_0 \in \mathcal M_2 \setminus \mathcal M_1$ with $\lambda_2^{(1)}\neq 0$ we have 
$$
\nu(\varepsilon) = \varepsilon^{k-1} \nu^{(k-1)} + \dots , \nu^{(k-1)} \neq 0
$$
and for a small deformation $s\to \alpha_s$ we shall still have $\lambda_2^{(1)} \neq 0,  \nu^{(k-1)} \neq 0$ and hence $\lambda_1^{(1)}=0$. 
Therefore $\mathcal M_2 \setminus \mathcal M_1$ is another irreducible component of the Nash space $ Arc(B_I\C^n,E)$. Similar considerations show that $\mathcal M_{k+1}$ is an union of exactly $k+1$ irreducible components 
 $\mathcal M_{i+1} \setminus \mathcal M_i$
of $ Arc(B_I\C^n,E)$, which are defined by the relations 
$$
 \mathcal M_{i+1} \setminus \mathcal M_i =  \{ (\lambda_1(\varepsilon), \lambda_2(\varepsilon)) :  \lambda_2^{(0)}=0, \; \lambda_1^{(0)}=\lambda_1^{(1)}=  \dots = \lambda_1^{(i)} = 0 , \lambda_1^{(i+1)} \neq 0 \}
$$
where $ i = 1, 2\dots  k-1$ and
\begin{align*}
 \mathcal M_{k+1} \setminus \mathcal M_k 
 =   \{ (\lambda_1(\varepsilon) ,\lambda_2(\varepsilon)) : \lambda_2^{(0)}=0, \; \lambda_1^{(0)}=  \dots = \lambda_1^{(k)} = 0 , (\lambda_1^{(k+1)}, \lambda_2^{(1)})  \neq (0,0)\} .
\end{align*}

 It is easily seen (by making use of the same deformation argument) that for $i>k+1$ there are no new components in $ \mathcal M_i$, so the irreducible decomposition of
$
 Arc(B_I\C^n,E)  
$
has exactly $k+1$ irreducible components
$$
 Arc(B_I\C^n,E)  = \overline{ \mathcal M_1} \cup \overline{\mathcal M_2 \setminus \mathcal M_1}\cup \dots \cup  \overline{ \mathcal M_{k+1} \setminus \mathcal M_k }
$$

\end{example}

As observed by Nash, it is a general fact that the arc space $Arc(B_I\C^n,E)$ has finitely many irreducible components. 
Thus, there exists $k$, such that the closure of $ \mathcal M_{k}$ is $Arc(B_I\C^n,E)$.
As the number $k$ is not known, the description of these irreducible components may be a formidable task, which is the content of the Nash problem. The description of  $\mathcal M_{k_{max}}$ , however,  will be sufficient for the purposes of the present paper.

Consider the canonical projection  map
$$
 \pi_I:  Arc(B_I\C^n,E) \to E : \alpha \mapsto \alpha(0) .
$$
which associates to an arc $\alpha$ on $B_I\C^n$ its center $\alpha(0)$. It is an algebraic map, and the image of each irreducible component of $ Arc(B_I\C^n,E)$ is a closed irreducible subset of $E$.
Thus, every point of $E$ is in the image of some irreducible component of the arc space, possibly in a non-unique way.
This motivates the following
\begin{definition}
\label{essentialset}
A  set $\mathcal M \subset  Arc(B_I\C^n,E)$ of irreducible
 components of $ Arc(B_I\C^n,E)$ is said to be essential, provided that
\begin{itemize}
\item
$\pi_I (\mathcal M) = E$
\item
$\mathcal M$ is minimal under inclusions.
\end{itemize}
\end{definition}

Although each component of $\mathcal M$ depends on infinitely many parameters, only a finite number of them are needed to specify the component, the other taking arbitrary complex (or real) values. 
This fact is especially important in the applications. For instance, in Example \ref{ex2}, we have $ \pi_I(\mathcal M_i) = [1:0]\in \mathbb P ^1$ for $i \leq k$ and
$ \pi_I(\mathcal M_{k+1}) =  \mathbb P ^1= E$. Therefore  the essential set $\mathcal M$ is irreducible and equal to $\overline{\mathcal M_{k+1} \setminus \mathcal M_k}$.
An element of $\mathcal M_{k+1} \setminus \mathcal M_k$ is written
$$
\lambda_1(\varepsilon) = \lambda_1^{(k+1)} \varepsilon ^{k+1} + \dots , \lambda_2(\varepsilon) = \lambda_{2}^{(1)} \varepsilon ^{1} + \dots, (\lambda_1^{(k+1)},\lambda_{2}^{(1)})\neq (0,0),
$$
and the dots stay for arbitrary power series $\sum_{i\geq k+2} \lambda_1^{(i)} \varepsilon^i$, $\sum_{i\geq 2} \lambda_{2}^{(i)} \varepsilon^i$. The coefficients of these series are non-essential in the sense that they are arbitrary and the corresponding center $\alpha(0) = [\lambda_1^{(k+1)}:\lambda_{2}^{(1)}]$ does not depend on them. This is not the case for the points on the border $\mathcal M \setminus \{\mathcal M_{(k+1)} \setminus \mathcal M_k\}$ for which $(\lambda_1^{(k+1)},\lambda_{2}^{(1)})= (0,0)$.

The notion of  "essential set" of irreducible components of the arc space $ Arc(B_I\C^n,E)$ is central for this paper,   in the next section it will appear under the term ``complete set of essential perturbations",  as introduced first by Iliev \cite{Iliev}.

\section{Blow up of the  Bautin ideal and the space of essential perturbations}
\label{section3}

In this section we describe a dictionary between the results of the preceding section, and the 16th Hilbert problem on a period annulus.
We use the notations of the Introduction.

\subsection{The Bautin ideal}
For an analytic real family of real analytic plane vector fields $X_\lambda$, such that $X_{\lambda_0} $ has a period annulus,
consider the displacement function (\ref{displacement}) 
$$
F(h,{\lambda })=\Sigma_{k=1}^{+\infty} F_k({\lambda })(h-h_0)^k,
$$
defined in section \ref{openannulus}. If we assume that it is analytic in a neighbourhood of a point $(\lambda_0,h_0)$ then the 
analytic functions $F_k= F_k(\lambda)$ define a germ of an ideal $\mathcal B _{\lambda_0}$ in the ring of germs of analytic functions $\mathcal O_{\lambda_0}$ at $\lambda_0\in \R^n$.  It is also clear, that the germs  $\mathcal B _{\lambda}$ extend  to some complex neighbourhood $U$ of $\lambda_0$, on which they define an ideal sheaf $\mathcal B(U)$ in the sheaf of analytic functions $\mathcal O (U)$. 
The ideal sheaf $\mathcal B (U)$ is called the Bautin ideal, associated to the family $X_{\lambda}$ on $U$, see \cite[[section 12]{ilya08} for details. 

\emph{In the present section we assume that $X_\lambda$ is the family of polynomial vector fields of degree at most $d$ with complex (or sometimes real) coefficients. }Our results are  easily adapted to the case, when the family depends only analytically in the parameters, or even the vector fields $X_\lambda$ are only analytic too.  The principal consequence of this assumption is,  that our results will be global. In particular the Bautin ideal will be polynomially generated. We can forget the origine of our problem and investigate the zeros of the displacement function $F(h,{\lambda })$ in regard to which we impose the following assumptions

\begin{itemize}
\item
There exists an open (in the complex topology) subset $U\subset \C^n$, on which the coefficients  $F_k$ of the displacement function define 
an ideal sheaf
$$
\mathcal B(U) = \cup_{\lambda\in U} \mathcal B _{\lambda} 
$$
\item
The ideal sheaf $\mathcal B(U)$ is polynomially generated : there exists a finite set of polynomials $v_1,v_2, \dots, v_N \in \C[\lambda_1,\lambda_2,\dots, \lambda_n]$, which generate the germ $ \mathcal B _{\lambda}, \forall \lambda \in U$.
\end{itemize}
\begin{definition}
The polynomial  ideal  
$$
{\bf B }= (v_1,v_2, \dots, v_N) \subset \C[\lambda_1,\lambda_2,\dots, \lambda_n]
$$
is called the Bautin ideal, associated to the ideal sheaf $\mathcal B(U)$ on the open set $U$.
\end{definition}

The main example of Bautin ideal in the above sense is given by the case of a  period annulus of center type of a family of polynomial vector fields of fixed degree. If the center point is placed at the origin, then the trace $v_1(\lambda)= Trace X_\lambda(0)$ is a section of the sheaf $\mathcal B(U)$. 
It is well known that on the variety $\{ \lambda\in U : v_1(\lambda)=0 \}$  the ideal sheaf $\mathcal B(U)$ is polynomially generated, say, by $v_2,\dots,v_N$ \cite{ilya08}. This shows that when
$\{v_1= 0 \}$ is a smooth divisor  (for instance $Trace X_\lambda(0)$ is linear in $\lambda$), then  $\mathcal B(U)$ is polynomially generated by 
$v_1,v_2, \dots, v_N$.

\subsection{Displacement function $F$ and the factors $(\Phi_1, \Phi_2, \dots ,\Phi_N)$}

Consider the displacement map in a neighbourhood of $h_0 \in\Sigma$, $\lambda_0 \in U$, defined by its Taylor expansion
$$
\begin{aligned}
F(h,{\lambda })=\Sigma_{k=1}^{+\infty} F_k({\lambda })(h-h_0)^k.
\end{aligned}
$$
As the ideal of coefficients $F_k$ generates the germ of Bautin ideal $\mathcal B_{\lambda_0}$ which is polynomially generated by  $v_1,v_2, \dots , v_N$, then
$$
\begin{aligned}
F(h,{\lambda })= \Sigma_{j=1}^N  v_j (\lambda) \Phi_j (h, \lambda ) .
\end{aligned}
$$
Consider an analytic deformation $X_{\lambda(\varepsilon)}$ in a neighbourhood of $\lambda_0$, where $X_{\lambda_0}$ is a vector field with a period annulus of closed orbits. Suppose further that $M_k$ is the   bifurcation function associated to this deformation, that is to say
$$
F(h,{\lambda(\varepsilon) }) = \varepsilon^k M_k(h) + O( \varepsilon^{k+1}) .
$$
\begin{proposition}
\begin{equation}
\label{bf1}
M_k(h) = \sum_{j=1}^N v_j^{(k)}  \Phi_j (h, \lambda_0 ) .
\end{equation}
where  $v_j^{(k)} $ are polynomials in the coefficients of the series $\lambda(\varepsilon)$, determined by the identities
\begin{equation}
\label{bf2}
v_j(\lambda(\epsilon))= \Sigma_{r\geq 0} v_{j}^{(r)}{\varepsilon}^r,
\end{equation}
and $k=k(\alpha)$ is the order of the arc $\alpha : \varepsilon \mapsto \lambda(\varepsilon)$, see Definition \ref{def1}.
\end{proposition}
\begin{remark}
The claim  generalizes the Bautin's fundamental lemma, as restated by 
C. Chicone and M. Jacobs  \cite[Lemma 4.1]{Chicone}, see also
(\cite{BGG}). Indeed, we do not suppose any special properties of the polynomial generators $(v_1,v_2, \dots , v_N)$. 
Note also that $h_0$ does not correspond to a 
 singular point (a center) of the vector field $X_{\lambda_0}$ which might not have a center at all (see Figure \ref{annuli}).
\end{remark}
\begin{proof}
 If
$$
\begin{aligned}
F(h,{\lambda })=\Sigma_{j=1}^{+\infty} u_j({\lambda })(h-h_0)^j,
\end{aligned}
$$
and $u_1,u_2, \dots , u_{N'}$ are generators of the localized Bautin ideal, then 
$$
F(h,{\lambda })=\Sigma_{j=1}^{N'} u_j({\lambda })\Psi_j (h, \lambda ) 
$$
and
$$
F(h,{\lambda(\varepsilon) }) = \varepsilon^{k'} \sum_{j=1}^{N'} u_j^{(k')}  \Psi_j (h, \lambda_0 ) + O( \varepsilon^{k'+1})
$$
where 
$$(u_1(\lambda(\varepsilon)),\dots ,
u_{N'}(\lambda(\varepsilon)))= \varepsilon^{k'} (1+ O(\varepsilon)) (u_1^{(k')},u_2^{(k')}, \dots, u_{N'}^{(k')}) 
$$
and $\Psi_j (h, \lambda_0 ) = (h-h_0)^j+\dots$.
According to Proposition \ref{blowupp}, the blowup does not depend on the generators. In particular if
$v_1,v_2, \dots, v_N$ is another set of generators, and
$$
(v_1(\lambda(\varepsilon)),\dots ,
v_N(\lambda(\varepsilon)))= \varepsilon^{k} (1+ O(\varepsilon)) (v_1^{(k)},v_2^{(k)}, \dots, u_{N}^{(k)})
$$
then $k=k'$, see (\ref{eq4}). Using the isomorphism (\ref{eq3}) we obtain the desired representation.

\end{proof}
The above Proposition has several implications. It allows to identify every bifurcation function $M$ of order $k$ to a point $ P\in E_{\lambda_0}$ by the correspondence
\begin{align}
\label{mel1}
M_k \to P=[v_1^{(k)},v_2^{(k)}: \dots : v_n^{(k)}] \in E_{\lambda_0} .
\end{align}
The opposite is also true : given a point $P\in E_{\lambda_0}$, by Proposition \ref{exceptional}, there is an arc from which we reconstruct the bifurcation function $M$, hence
\begin{corollary}
The projectivized set of bifurcation functions associated to one - parameter deformations $X_{\lambda(\varepsilon)}$ of the vector field $X_{\lambda_0}$ is in bijective correspondence with the points on the exceptional divisor $E_{\lambda_0}$. This correspondence is a linear isomorphism (by Proposition \ref{blowupp} (ii))
\end{corollary}

Let $V_\lambda\subset \C^N$ be the vector space spanned by the set of points on the exceptional divisor $E_\lambda \subset \mathbb P^{N-1}$ in $\C^N$.
\begin{corollary}
The space of all bifurcation functions associated to deformations $X_{\lambda(\varepsilon)}$ of the vector field $X_{\lambda_0}$ span a vector space of dimension $\dim V_{\lambda_0}$.
\end{corollary}

As we already noted in the preceding section,  $\pi_I (\mathcal M_{k_{max}}) = E$, see also Definition \ref{essentialset}.
This implies

\begin{corollary}
The minimal order of every bifurcation functions associated to $\lambda_0$ is bounded by $k_{\lambda_0} = \sup_{\pi(P)=\lambda_0} k_P$, and the number
$\sup_{\lambda\in U} k_{\lambda}$ is finite.
\end{corollary}

Definition \ref{essentialset} can be reformulated as follows
\begin{definition}
\label{essentialperturbation}
A complete set of essential perturbations of a period annulus is a set $\mathcal M$ of  one-parameter deformations $X_{\lambda(\varepsilon)}$ 
which, viewed as  arcs, form an essential  set of irreducible components of the   Nash space of arcs $  Arc(B_{I}\C^n,E)$. Each irreducible component of $\mathcal M$ is referred to as an essential perturbation.
\end{definition}

In his seminal paper, Iliev 
describes 
\begin{quote}
``\emph{A set of essential perturbations which can realize the maximum number of limit cycles produced by the whole 
class of quadratic systems}" \cite[page 22]{Iliev}. 
\end{quote}
The property ``to realize the maximum number of limit cycles", however, is a consequence of the fact, that the selected set of essential perturbations produces all possible bifurcation functions.  To describe a set of essential perturbations, we need to study the Nash arc space $Arc( B_{I}\C^n, E)$ and select appropriate essential irreducible components. 
It is also clear from Definition \ref{essentialset} that the number of essential perturbations is not smaller than the number of irreducible components of the center variety $V(I)$.
As we shall see in the next section, it can be strictly bigger: the number of the irreducible components of the center set of quadratic vector fields is four, but the number of essential perturbations is five.

For convenience of the reader we resume the correspondence between section \ref{section2} and section \ref{section3}
in the following  table. 
\vspace{1cm}

\begin{tabular}{|c|c|}
  \hline
 arcs in algebraic geometry & bifurcation theory  \\
  \hline
  parameter space $\{\lambda \}$ & space of plane vector fields $X_\lambda$   \\
  ideal $I$ &Bautin ideal $\bf B$   \\
  variety $V(I)$ & center set  \\
   Blow up $B_I\C^n$ of an ideal $I$ &   Blow up of the Bautin ideal \\  
   arc on  the  blow up $B_I\C^n$  & one-parameter deformation    $X_{\lambda(\varepsilon)}$  \\
   order of an arc & order of a bifurcation function \\
   point on the exceptional divisor & bifurcation function associated to $X_{\lambda(\varepsilon)}$\\
   exceptional divisor & set of all bifurcation functions \\
   essential set of irreducible components \\ of the Nash space & complete set of essential perturbations \\ 
  \hline
\end{tabular}

\section{Quadratic centers and Iliev's essential perturbations}
\label{quadratic}

We revisit Iliev's computations of essential perturbations with emphasis on Nash spaces of arcs, as explained in the previous two sections.
We focus on the Kapteyn normal form $X_\lambda$  (\ref{BautinNF}) of quadratic systems, which we recall briefly.

A quadratic vector field near a center is conveniently written in complex notations $z=x+{\rm i}y$, see   \cite[Zoladek]{Zoladek1}:

\begin{equation}\label{CQuad}
\dot{z}=({\rm i}+\lambda)z+Az^2+B\mid z\mid^2+C\overline{z}^2.
\end{equation}
with $\lambda, x, y \in\mathbb{R},(A,B,C)\in\mathbb{C}^3$. The underlying real parameters of the planar vector field are $\lambda,a,a',b,b',c,c'$ :

\begin{eqnarray}
\begin{aligned}\label{Quad}
\dot{x}=& \;{\lambda} x-y+ax^2+bxy+cy^2,\\
\dot{y}= &\; x+{\lambda} y+a'x^2+b'xy+c'y^2,
\end{aligned}
\end{eqnarray}
with the linear relations:

\begin{align*}
a+{\rm i}a' &=A+B+C\\
b+{\rm i}b'&=2{\rm i}(A-C)\\
 c+{\rm i}c'&=-A+B-C,\\
A&=\frac{1}{4}[a-c+b'+{\rm i}(a'-c'-b)]\\
 B&=\frac{1}{2}[a+c+{\rm i}(a'+c')] \\
  C&=\frac{1}{4}[a-c-b'+{\rm i}(a'-c'+b)].
\end{align*}

With these variables the Bautin ideal is generated by the four polynomials (with real coefficients):

\begin{align*}
v_1&=\lambda\\
v_2&={\rm Im}(AB)\\
v_3&={\rm Im}[(2A+\overline{B})(A-2\overline{B})\overline{B}C]\\
v_4&={\rm Im}[(\mid B\mid^2-\mid C\mid^2)(2A+\overline{B})\overline{B}^2C].
\end{align*}

The components are then given by:

\begin{eqnarray}
\label{strata}
\begin{aligned}
LV:  & \;\lambda=B=0\\
R\;:& \;\lambda={\rm Im}(AB)={\rm Im}(\overline{B}^3C)={\rm Im}(A^3C)=0\\
H\::& \;\lambda=2A+\overline{B}=0\\
Q_4:&\; \lambda=(A-2\overline{B})=(\mid B\mid-\mid C\mid)=0.
\end{aligned}
\end{eqnarray}

The above computation goes back essentially to Dulac and Kapteyn, see \cite{Zoladek1},  \cite{DSchlomiuk}.  
 The usual terminology in the real case  is, according to (\ref{strata}) : Hamiltonian $H$, reversible (or symmetric) $R$, Lotka-Volterra $LV$, and co-dimension four (or Darboux) $Q_4$ component of the center set, respectively.
Another terminology is introduced in  \cite[section 13]{ilya08}. \vskip 1pt

If we assume $B\neq 0$, performing a suitable rotation and scaling of coordinates, we can suppose $B=2$. Similarly if $B=0$ but $A\neq 0$, we take $A=1$ (LV), and when $A=B=0$, we take $C=1$ (Hamiltonian triangle). In the case where $B=2$, there is a center if and only if the following conditions hold: (i) $A=-1$ (H), (ii) $A=a$ and $C=b$ are real (R), (iii) $A=4$, $\mid C\mid=2$ (Codimension 4).

The list of quadratic centers looks hence as follows:

\begin{itemize}
\item $\dot{z}=-{\rm i}z-z^2+2\mid z\mid^2+(b+{\rm i}c)\overline{z}^2,$ Hamiltonian (H)
\item $\dot{z}=-{\rm i}z+az^2+2\mid z\mid^2+b\overline{z}^2,$ Reversible (R)
\item $\dot{z}=-{\rm i}z+z^2+(b+{\rm i}c)\overline{z}^2,$ Lotka-Volterra (LV)
\item $\dot{z}=-{\rm i}z+4z^2+2\mid z\mid^2+(b+{\rm i}c)\overline{z}^2, \mid b+{\rm i}c\mid=2,$ Codimension 4 $(Q_4)$
\item $\dot{z}=-{\rm i}z+\overline{z}^2,$ Hamiltonian triangle.
\end{itemize}
 
We can observe that up to a rotation and scaling of coordinates, $H$, $R$ and $LV$ can be represented by planes and the ``Codimension 4" stratum by a quadric (cf. figure 2).
\vskip 1pt
In fact, centers from (H) with $c=0$ are also reversible. They belong to the intersection $(H)\cap (R)$ and can also be defined in $(R)$ by $a=1$. Note that centers from $(Q_4)$ such that $c=0, b=\pm 2$ are also reversible. The Lotka-Volterra centers so that $c=0$ are also reversible. They form, together with the Hamiltonian triangle the {\it degenerate} centers:

\begin{itemize}
\item $\dot{z}=-{\rm i}z-z^2+2\mid z\mid^2+b\overline{z}^2,$ Reversible Hamiltonian $(H)\cap (R)$ 
\item $\dot{z}=-{\rm i}z+z^2+b\overline{z}^2,$ Reversible Lotka-Volterra $(LV)\cap (R)$
\item $\dot{z}=-{\rm i}z+4z^2+2\mid z\mid^2+b\overline{z}^2, b=\pm 2,$ $(Q_4)\cap (R)$
\item $\dot{z}=-{\rm i}z+\overline{z}^2,$ $(LV)\cap (H)$, Hamiltonian triangle.
\end{itemize}
\vskip 1pt

The Kapteyn normal form of the quadratic vector fields near a linear center is:

\begin{eqnarray}\label{BautinNF}
X_\lambda :\left\{\begin{aligned}
\dot{x}={\lambda}_1 x-y-{\lambda}_3x^2+(2\lambda_2+\lambda_5)xy + {\lambda}_6y^2,\\
\dot{y}=x+{\lambda}_1 y+{\lambda}_2x^2+(2\lambda_3+\lambda_4)xy-{\lambda}_2y^2.
\end{aligned}
\right.
\end{eqnarray}
It provides indeed a local affine chart of the space of quadratic vector fields (\ref{CQuad}) modulo the action of $\mathbb{C}^*$.
The Kapteyn normal form, although simple, can be misleading, as pointed out first by Zoladek \cite[Remark 2, page 238]{Zoladek1}.
One should be very careful that they can be compared with the previous complex parameters only under the extra condition ${\rm Im}(B)=0$. Nevertheless, it is more convenient to use these Kapteyn parameters to compute effectively  the corresponding space of arcs and jets.

The relation between Kapteyn's coefficients and the previous complex coefficients are given by:

\begin{align*}
A&=(\lambda_3-\lambda_6+\lambda_4-{\rm i}\lambda_5)/4\\
B&=(\lambda_6-\lambda_3)/2\\
C&=[-(3\lambda_3+\lambda_6+\lambda_4)+(4\lambda_2+\lambda_5){\rm i}]/4.
\end{align*}

For this choice of parameters, the Bautin ideal $\bf B \subset \mathbb K[\lambda_1,\lambda_2,\lambda_3,\lambda_4,\lambda_5,\lambda_6]$, where $\mathbb K = \C$ or $\R$, is generated by

\begin{eqnarray}
\label{kapteynideal}
\begin{aligned}
v_1(\lambda)&=\lambda_1,\\
v_2(\lambda)&=\lambda_5(\lambda_3-\lambda_6),\\
v_3(\lambda)&=\lambda_2\lambda_4(\lambda_3-\lambda_6)(\lambda_4+5\lambda_3-5\lambda_6),\\
v_4(\lambda)&=\lambda_2\lambda_4(\lambda_3-\lambda_6)^2(\lambda_3\lambda_6-2\lambda_6^2-\lambda_2^2)
\end{aligned}
\end{eqnarray}

The affine algebraic variety defined by this ideal is denoted $Z= Z({\bf B})$ and  called the quadratic centre set. The variety displays four irreducible components 
 $Z= I_1\cup I_2\cup I_3\cup I_4$ : three planes and one affine quadratic cone. When the base field is $\R$ their mutual position in $\R^6$ is shown on fig. \ref{centerset1}. Note that $I_1\cap I_2$ , $I_2\cap I_3$, $I_3 \cap I_1$ are two-planes, while $I_1\cap I_2 \cap I_3$ is a straight line.
 \begin{eqnarray}
\label{kapteynzeroset}
\begin{aligned}
I_1&:\lambda_1=0, \lambda_3-\lambda_6=0\\
I_2&:\lambda_1=0, \lambda_2=0, \lambda_5=0\\
I_3&:\lambda_1=0, \lambda_4=0, \lambda_5=0\\
I_4&:\lambda_1=0, \lambda_5=0, \lambda_4+5\lambda_3-5\lambda_6=0, \lambda_3\lambda_6-2\lambda_6^2-\lambda^2_2=0.
\end{aligned}
\end{eqnarray}

\begin{figure}
\includegraphics[width=10cm]{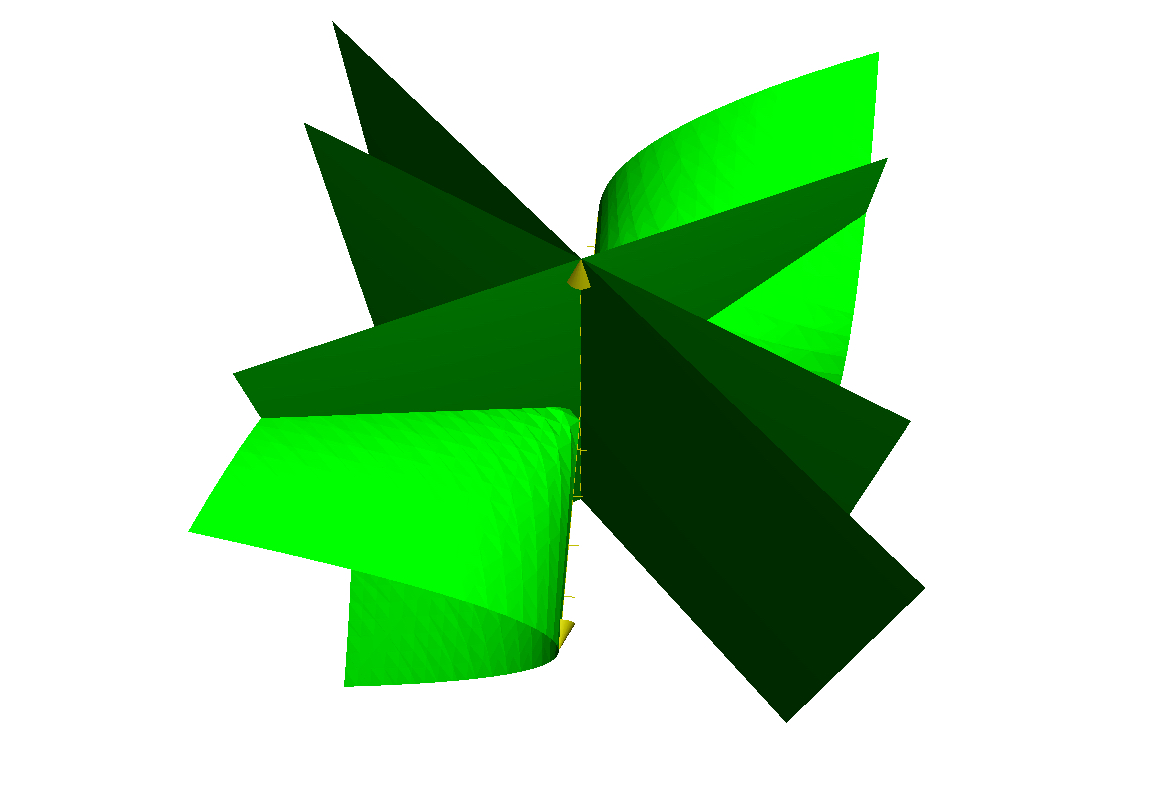}
\caption{Sketch of the mutual position in $\R^6$ of the irreducible components $I_1,I_2,I_3, I_4$ of the real center set (\ref{kapteynzeroset})}
\label{centerset1}
\end{figure}

These  irreducible components coincide with $LV$, $R$, $H$ and $Q_4$ but only under the extra-condition  $B\in \R^*$. Indeed,  a vector field which belongs to $LV$ and $H$, so that $A=B=0$ is necessarily in $R$ (it is the so-called Hamiltonian triangle). Still, if it belongs to $I_1\cap I_3$, then $\lambda_3=\lambda_6, \lambda_4=\lambda_5=0$ but $\lambda_2$ is not necessarily equal to $0$ and it does not necessarily belongs to $I_2$.

We consider the blowup 
of the Bautin ideal ${\bf B}$ which is the graph  of the map
\begin{equation}
\begin{aligned}
\C^6 &\to \mathbb P^3\\
\lambda & \mapsto [v_1:v_2:v_3:v_4] 
\label{blowup1}
\end{aligned}
\end{equation}
with projection
$$
\pi : \C^6 \times \mathbb P^3  \to \C^6
$$
and exceptional divisor $E= \pi^{-1}(Z({\bf B}))$, see  (\ref{blowup}). As $E$ is identified to the set of bifurcation functions, and arcs to one-parameter perturbations $X_{\lambda(\varepsilon)}$,
 we construct a set of essential perturbation in the sense of Definition \ref{essentialperturbation} and Definition \ref{essentialset}.
This computation  breakes into two steps:
\begin{itemize}
\item 
find a minimal list of  families of perturbations $X_{\lambda(\varepsilon)}$ which project under $\pi_I$ to the exceptional divisor $E$, see Definition \ref{essentialset}
\item
check whether the essential perturbations are irreducible components of the corresponding Nash arc space
\end{itemize}
The first step,  can be found in \cite[Theorem 6]{BGG}, where ten families of perturbations were identified, which produce all possible bifurcation functions.  
As we shall see below, only five of them are irreducible components of the Nash space, the others are contained in their closures. These essential perturbations are denoted below 
$$Arc(I_1), Arc(I_2), Arc(I_3), Arc(I_4),  Arc(I_1\cap  I_3)$$
 and correspond respectively to the Lotka-Volterra, reversible, Hamiltonian, co-dimension four, and Hamiltonian triangle strata of the centre set. 
 This remarkable geometric fact has an analytic counterpart : 
 only bifurcation functions corresponding to the five irreducible components of the Nash space have to be considered, as far as we can obtain all the others by taking suitable limits.
 For instance, the bifurcation functions corresponding to the arc space $Arc(I_1 \cap  I_2)$ (reversible Lotka-Volterra systems) can be obtained as appropriate  limits of bifurcation functions associated to $Arc(I_1)$,
 in contrast to bifurcation functions associated to $Arc(I_1\cap  I_3)$ which can not be obtained as such limits. The reason is that 
 $Arc(I_1 \cap I_2)$ is in the closure of $Arc(I_1)$, while $Arc(I_1\cap  I_3)$ is a distinct irreducible component of the Nash arc space of the blowup of the Bautin ideal (and hence is not in the closure of $Arc(I_1)$). 
This confluence phenomenon was observed by Iliev in \cite[Corollary 1]{Iliev}.

\subsection{Smooth points of the center set $Z({\bf B}) $ } 
The base field in this section is $\C$.
It is straightforward to check that a point on the centre set
$
Z({\bf B}) = I_1\cup I_2 \cup I_3 \cup I_4
$
(\ref{kapteynzeroset})
is smooth, if and only if  it belongs to some $I_i$ but does not belong to $I_i\cap I_j$, $j\neq i$.  Smooth points were called ``generic" in \cite{Iliev}, and we denote this set $Z({\bf B}) _{reg}$.

Let 
$\lambda^* \in Z({\bf B})_{reg} $. The localized Bautin ideal $\mathcal B_{\lambda^*} $ defined by the polynomials  (\ref{kapteynideal}) is radical and generated by
\begin{itemize}
\item $\lambda_1, \lambda_3-\lambda_6$ if $\lambda^* \in I_1 $
\item $\lambda_1, \lambda_2,\lambda_5$ if $\lambda^* \in I_2 $ 
\item
$\lambda_1, \lambda_4, \lambda_5$ 
or  $\lambda^* \in I_3 $
\item $\lambda_1,   \lambda_5, \lambda_4+5\lambda_3-5\lambda_6, \lambda_3\lambda_6-2\lambda_6^2-\lambda^2_2$
 if $\lambda^* \in I_4 $ .
\end{itemize}
The exceptional divisor $E_{\lambda^*}$ is isomorphic to the corresponding exceptional divisor of one of  the blowups
\begin{align*}
\lambda & \to [\lambda_1: \lambda_3-\lambda_6] \\
\lambda & \to [\lambda_1: \lambda_2:\lambda_5]  \\
\lambda & \to [\lambda_1: \lambda_4:\lambda_5]  \\
\lambda & \to  [\lambda_1:  \lambda_5: \lambda_4+5\lambda_3-5\lambda_6: \lambda_3\lambda_6-2\lambda_6^2-\lambda^2_2 ]
\end{align*}
It is therefore straightforward to compute $E_{\lambda^*}$, and it turns out that it is a projective space, see table \ref{table1},
\begin{table}[htp]
\begin{center}
\begin{tabular}{|c|c|c|c|}
\hline
$\lambda^* \in I_1 $ & $\lambda^* \in I_2 $ & $\lambda^* \in I_3 $ & $\lambda^* \in I_4 $\\
\hline
$E_{\lambda^*} = \mathbb P^1$ & $E_{\lambda^*} = \mathbb P^2$ & $E_{\lambda^*} = \mathbb P^2$ &$ E_{\lambda^*} = \mathbb P^3$ \\
\hline
\end{tabular}
\end{center}
\caption{The exceptional divisor $E_{\lambda^*}$ when $\lambda^*$ is a smooth point.}
\label{table1}
\end{table}%
so the set of bifurcation functions is a vector space of dimension 
$2,3, 3$ and $4$, which is also the co-dimension of $I_1,I_2, I_3, I_4$  respectively. 
 The center set at $\lambda^*$, and its blow up along $E_{\lambda^*}$, are smooth, 
   so the arc space is easy to describe.
 An element of the Nash arc space $Arc( B_I\C^n,E_{\lambda^*})$, $I=\mathbf B$, $\lambda^*\in  Z({\bf B})_{reg} $ is an arc
\begin{eqnarray}
\label{arc11}
\begin{aligned}
\varepsilon \to \left( \lambda(\varepsilon), [v_1(\lambda(\varepsilon)):v_2(\lambda(\varepsilon)): \dots : v_6(\lambda(\varepsilon)) ] \right) \in 
\C^6 \times \mathbb P^5\\
\lambda(0)= \lambda^*
\end{aligned}
\end{eqnarray}
and it corresponds to the one-parameter family of vector fields
$
X_{\lambda(\varepsilon)} $.
The arc space $Arc( B_{\bf B}\C^n,E_{\lambda^*})$ has only one irreducible component corresponding to the irreducible smooth divisor $E_{\lambda^*}$. An essential family of arcs is  a family parameterized by $E_{\lambda^*}$ and having a minimal number of parameters. Thus, as essential family of arcs (\ref{arc11}) we can take 
\begin{equation}
\label{smooth}
\lambda(\varepsilon)= \lambda^* + \lambda^{(1)} \varepsilon, \; v_i=v_i(\lambda(\varepsilon))
\end{equation}
where
\begin{itemize}
\item $\lambda^{(1)} = (\lambda_{1,1},0,0,0,0,\lambda_{6,1})$ if $\lambda^* \in I_1 $
\item $\lambda^{(1)} = (\lambda_{1,1},\lambda_{2,1},0,0,\lambda_{5,1},0)$ if  $\lambda^* \in I_2 $
\item $\lambda^{(1)} = (\lambda_{1,1},0,0,\lambda_{4,1},,\lambda_{5,1},0)$ if  $\lambda^* \in I_3 $
\item $\lambda^{(1)} = (\lambda_{1,1},\lambda_{2,1},0,\lambda_{4,1},\lambda_{5,1},0)$ if  $\lambda^* \in I_4 $ .
\end{itemize}
The corresponding essential perturbations of the integrable quadratic vector field $X_{\lambda^*}$ take the form as in \cite[Theorem 6]{BGG}
\begin{itemize}
\item  $X_{\lambda^*} + \varepsilon \lambda_{1,1} (x \frac{\partial}{\partial x} +  y\frac{\partial}{\partial y})   +  \varepsilon {\lambda}_{6,1} y^2\frac{\partial}{\partial x} $
 if $\lambda^* \in I_1 $
\item
$X_{\lambda^*} + \varepsilon \lambda_{1,1} (x \frac{\partial}{\partial x} +  y\frac{\partial}{\partial y}) 
 + \varepsilon (2\lambda_{2,1} + \lambda_{5,1}) xy \frac{\partial}{\partial x}  + \varepsilon \lambda_{2,1}  (x^2-y^2) \frac{\partial}{\partial y}  
$
if  $\lambda^* \in I_2 $
\item 
$X_{\lambda^*} + \varepsilon \lambda_{1,1} (x \frac{\partial}{\partial x} +  y\frac{\partial}{\partial y}) +
 \varepsilon \lambda_{5,1} xy \frac{\partial}{\partial x}  +
 \varepsilon \lambda_{4,1}  xy \frac{\partial}{\partial y} 
$
 if  $\lambda^* \in I_3 $
\item 
$
X_{\lambda^*} + \varepsilon \lambda_{1,1} (x \frac{\partial}{\partial x} +  y\frac{\partial}{\partial y}) +
 \varepsilon (2 \lambda_{2,1} +\lambda_{5,1}) xy \frac{\partial}{\partial x} + 
 \varepsilon \lambda_{4,1} xy \frac{\partial}{\partial y} 
$
 if  $\lambda^* \in I_4 $ .
\end{itemize}
and the maximal order of the bifurcation function $M_k$ is $1$. 
We stress on the fact, that each of the above four essential perturbations depends upon six parameters, given by $\lambda^{(1)}$ and by $\lambda^*$.  The center $\lambda^*$ is therefore not fixed, but belongs to the corresponding smooth stratum $I_j$.

Finally, the deformations  (\ref{smooth}) form an irreducible component of the arc space, see Proposition \ref{mk}. Indeed,
it is obvious that a small deformation of (\ref{smooth}) leads to a family of the same form, under the condition that $\lambda^*$ is a smooth point. We conclude that the above families of vector fields  are irreducible components of the Nash arc space, which we denote for brevity
$$
Arc(I_1), Arc(I_2), Arc(I_3), Arc(I_4) .
$$

\subsection{Non-smooth points  of the center set $Z({\bf B})$ }
The base field in this section will be $\R$. 
The singular part $Z({\bf B})_{sing}$ of the real centre set has five irreducible components given by three co-dimension four planes 
 $I_1\cap I_2, I_1\cap I_3$ and $I_2\cap I_3$, and the reducible set
$I_4\cap I_2$. The latter is a union of two straight lines intersecting at the origin. As $I_2$ is a three-plane, then  the vector  spaces $I_1\cap I_2 , I_2\cap I_3, I_2 \cap I_4$ can be represented in $\R^3=I_2$ as on 
  fig. \ref{figuresat}. 
 $Z({\bf B})_{sing}$ is smooth, except along the line $I_1\cap I_2\cap I_3$. Thus, $Z({\bf B})_{sing}$ is a disjoint union of five smooth varieties 
 $$
 \{ I_1 \cap I_2 \} \setminus  \{ I_1\cap I_2 \cap I_3\}, \{ I_2 \cap I_3 \} \setminus \{ I_1\cap I_2 \cap I_3\}, \{ I_3 \cap I_1 \} \setminus  \{I_1\cap I_2 \cap I_3\} $$
 $$
  \{  I_1\cap I_2 \cap I_3\} \setminus \{ 0 \}, \{ 0 \} 
 $$
 which we consider separately. For brevity, and if there is no confusion,  we shall denote each of the above sets by 
 $$
 \{ I_1 \cap I_2 \}  , \{ I_2 \cap I_3 \}  , \{ I_3 \cap I_1 \}  ,
  \{  I_1\cap I_2 \cap I_3\}  , \{ 0 \} .
 $$
The exceptional divisor $E_{\lambda^*}$ in each of the five cases is presented on table \ref{table2} (this straightforward computation is omitted).
\begin{table}[htp]
\begin{center}
\begin{tabular}{|l|c|r| }
\hline
$\lambda^* \in I_1\cap I_3 $ & $\lambda^* \in I_2\cap I_1 $ & $\lambda^* \in I_2\cap I_3 $ \\
\hline
$E_{\lambda^*} = \mathbb P^3$ & $E_{\lambda^*} = \mathbb P^2$ & $E_{\lambda^*} = \mathbb P^2$    \\
\hline
\end{tabular}
\begin{tabular}{|l|c|r| }
\hline
 $\lambda^* \in I_2 \cap I_4$ & $\lambda^* \in I_1\cap I_2 \cap I_3 $ &$ \lambda^*=0  $\\
\hline
$ E_{\lambda^*} = \mathbb P^3$  & $ E_{\lambda^*} = \mathbb P^3$ &$E_{\lambda^*} = \mathbb P^3$    \\
\hline
\end{tabular}
\end{center}
\caption{The exceptional divisor $E_{\lambda^*}$ when $\lambda^* \in Z({\bf B})_{sing}$.}
\label{table2}
\end{table}%

\begin{figure}
\includegraphics[width=10cm]{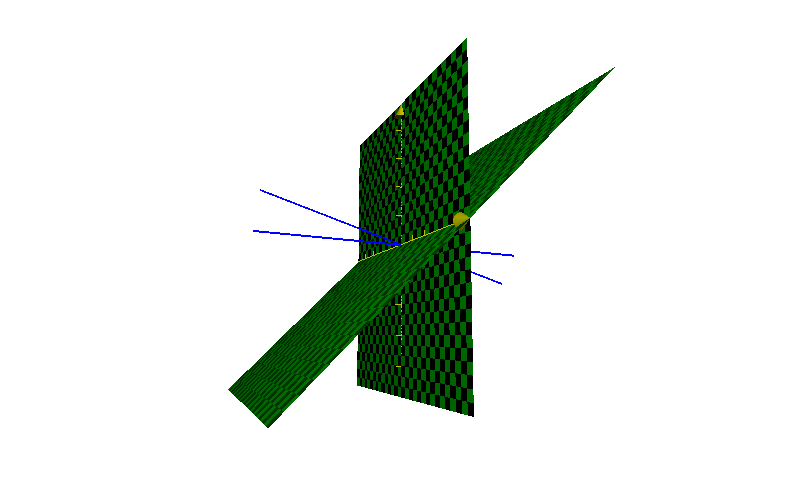}
\caption{The four irreducible components of the singular set  $\{I_2\cap I_1\} \cup \{I_2\cap I_3\} \cup \{I_2\cap I_4\} $  represented in $ \R^3=I_2$ .}
\label{figuresat}
\end{figure}

\subsubsection{ The essential perturbations of the center set $I_1\cap I_3$.} This is probably the most interesting case, for this reason we give more details.
Let $\lambda^*$ be a smooth point on the two-plane $I_1\cap I_3$ (that is to say, $\lambda^*\not\in  I_1\cap I_2 \cap I_3$).
As $\lambda_2^*\neq 0$, then the
 localized Bautin ideal $\mathcal B_{\lambda^*} $ defined by the polynomials  (\ref{kapteynideal}) is also generated by  
\begin{equation}
\label{v13}
 \lambda_1, \lambda_5(\lambda_3 - \lambda_6), \lambda_4^2 (\lambda_3-\lambda_6),  \lambda_4 (\lambda_3-\lambda_6)^2 .
\end{equation}
which will be therefore used on the place of $v_1,v_2,v_3,v_4$, in order to blow up  $\mathcal B_{\lambda^*} $. 
To the end of the section, instead of (\ref{blowup1}), we consider the blowup
\begin{align}
\lambda \to & [\lambda_1: \lambda_5(\lambda_3 - \lambda_6): \lambda_4^2 (\lambda_3-\lambda_6):  \lambda_4 (\lambda_3-\lambda_6)^2 ]
\end{align}
Using Proposition \ref{exceptional}  we verify first, that  $E_{\lambda^*} = \mathbb P^3$ and hence the vector space of bifurcation functions is of dimension $4$.
A general arc centred at a general point  $P\in E_{I_1\cap I_3}$, $E_{I_1\cap I_3} = \cup_{\lambda\in I_1\cap I_3} E_\lambda$, is defined by (\ref{arc1}), where $\lambda^* \in I_1\cap I_3$ and
\begin{eqnarray}
\label{arc1}
\begin{aligned}
\lambda_1&= \varepsilon^3  \lambda_{1,3} + O(\varepsilon^4) \\
\lambda_5 &= \varepsilon^2 \lambda_{5,2}  + O(\varepsilon^3)\\
\lambda_3-\lambda_6 &=  \varepsilon  (\lambda_{3,1}-\lambda_{6,1})+ O(\varepsilon^2)\\
\lambda_4 &=\varepsilon  \lambda_{4,1} + O(\varepsilon^2) 
\end{aligned}
\end{eqnarray}
with center
$$P=
[\lambda_{1,3}: \lambda_{5,2}(\lambda_{3,1} - \lambda_{6,1}): \lambda_{4,1}^2 (\lambda_{3,1}-\lambda_{6,1}):  \lambda_{4,1} (\lambda_{3,1}-\lambda_{6,1})^2 ] .
$$
Clearly, $P$ can take any value on $ \mathbb P^3$  except $[*:*:1:0]$ and $ [*:*:0:1]$. The missing point  $ [*:*:0:1]$ can be obtained as centre of the following arc
\begin{eqnarray}
\label{arc2}
\begin{aligned}
\lambda_1&=  \varepsilon^4\lambda_{1,4} + O(\varepsilon^5) \\
\lambda_5 &= \varepsilon^3 \lambda_{5,3} + O(\varepsilon^4)\\
\lambda_3-\lambda_6 &=  \varepsilon (\lambda_{3,1}-\lambda_{6,1})+ O(\varepsilon^2)\\
\lambda_4 &= \varepsilon^2  \lambda_{4,2}+ O(\varepsilon^3)
\end{aligned}
\end{eqnarray}
where
$$P=
[\lambda_{1,4}: \lambda_{5,3}(\lambda_{3,1} - \lambda_{6,1}): 0 :  \lambda_{4,2} (\lambda_{3,1}-\lambda_{6,1})^2 ] .
$$
Similar considerations are valid of course for the centres $[*:*:1:0]$, which shows  that  $E_{\lambda^*} = \mathbb P^3$. 
 
 It is easy to see,  that the family of arcs (\ref{arc1}) is an irreducible component of the arc space. Indeed, under a small deformation such that $\lambda^*\in I_1\cap I_3$, the degree of $\lambda_i(\varepsilon)$ neither decreases (the point $P$ is general) neither increases (because of Proposition \ref{mk}). We do not leave the family of arcs (\ref{arc1}). If we allow $\lambda^* \in I_3, \lambda\not\in I_1$ then we note that the dimension of $E_{\lambda^*}$ drops by one, and taking a limit $\lambda^* \to I_1\cap I_3$
  we can not obtain $E_{\lambda^*} = \mathbb P^3$.

 Therefore (\ref{arc1}) is an irreducible component of the arc space 
 $$Arc( B_{\bf B},E_{I_1\cap I_3}), E_{I_1\cap I_3} = \cup_{\lambda\in I_1\cap I_3} E_\lambda .
 $$
 The family
 (\ref{arc2}), however, is not an irreducible component, as it belongs to the closure of the family (\ref{arc1}).
To see this we need to show that every arc  (\ref{arc2}) is a continuous limit of arcs of the form (\ref{arc1}). Recall that we deform arcs on the blowup surface 
$Arc( B_{\bf B},E_{\lambda^*})$, see (\ref{arc11}), and continuity 
of the  deformation means that the coefficients depend\emph{ analytically} upon the deformation parameters \cite{joko16, LR}.
Consider now the family of arcs of the type (\ref{arc3}), parameterized by $\delta\neq 0$
\begin{eqnarray}
\label{arc3}
\begin{aligned}
\lambda_1&= \varepsilon^3(\varepsilon + \delta) (\lambda_{1,3} + \dots )\\
\lambda_5 &=  \varepsilon^2(\varepsilon + \delta)(\lambda_{5,2} + \dots )\\
\lambda_3-\lambda_6 &= \varepsilon  (\lambda_{3,1}-\lambda_{6,1}+ \dots ) \\
\lambda_4 &= \varepsilon(\varepsilon + \delta) (\lambda_{4,1} + \dots )
\end{aligned}
\end{eqnarray}
where the dots replace some convergent series vanishing for $\varepsilon = 0$. It follows from (\ref{arc3}) that
\begin{align*}
 &[\lambda_1: \lambda_5(\lambda_3 - \lambda_6) : \lambda_4^2 (\lambda_3-\lambda_6):  \lambda_4 (\lambda_3-\lambda_6)^2 ] = \\
= & \; [\lambda_{1,3}   +O(\varepsilon) : \lambda_{5,2}(\lambda_{3,1} - \lambda_{6,1})  +O(\varepsilon) \\
& : \lambda_{4,1}^2 (\lambda_{3,1}-\lambda_{6,1})\delta   +O(\varepsilon) :  \lambda_{4,1} (\lambda_{3,1}-\lambda_{6,1})^2    +O(\varepsilon) ] 
\end{align*}
where $O(\varepsilon) $ is also analytic in $\delta$ and $\lambda_{i,j}$. This family of arcs depends continuously on $\delta$ in the topology of the arc space and hence the limit $\delta\to 0$ can be taken. The center of (\ref{arc3}) is the point
$$
P=
[\lambda_{1,3}  : \lambda_{5,2}(\lambda_{3,1} - \lambda_{6,1})  : \lambda_{4,1}^2 (\lambda_{3,1}-\lambda_{6,1})\delta  :  \lambda_{4,1} (\lambda_{3,1}-\lambda_{6,1})^2 ] .
$$
which tends to
$$P=
[\lambda_{1,3}: \lambda_{5,2}(\lambda_{3,1} - \lambda_{6,1}): 0 :  \lambda_{4,1} (\lambda_{3,1}-\lambda_{6,1})^2 ] .
$$
as $\delta$ tends to $0$. Thus any arc of type (\ref{arc2}) can be obtained as a continuous limit of an arc of type (\ref{arc1}), and hence belongs to the same irreducible component of the arc space.

To resume, the above considerations show that the families
 of  vector fields $X_{\lambda^*}$, $\lambda^* \in I_1\cap I_3$   :
\begin{itemize}
\item  
$X_{\lambda^*} + 
\varepsilon^3 \lambda_{1,3} (x \frac{\partial}{\partial x} +  y\frac{\partial}{\partial y})  \\
+ \varepsilon^2 \lambda_{5,2} xy \frac{\partial}{\partial x}
+ \varepsilon   \lambda_{6,1}y^2 \frac{\partial}{\partial x}
+ \varepsilon  \lambda_{4,1} xy \frac{\partial}{\partial y}$
\item
$X_{\lambda^*} + 
\varepsilon^4 \lambda_{1,4} (x \frac{\partial}{\partial x} +  y\frac{\partial}{\partial y})  \\
+ \varepsilon^3\lambda_{5,3} xy \frac{\partial}{\partial x}
+ \varepsilon   \lambda_{6,1}y^2 \frac{\partial}{\partial x}
+ \varepsilon^2 \lambda_{4,2} xy \frac{\partial}{\partial y}$
\item
$X_{\lambda^*} + 
\varepsilon^4 \lambda_{1,4} (x \frac{\partial}{\partial x} +  y\frac{\partial}{\partial y})  \\
+ \varepsilon^2\lambda_{5,2} xy \frac{\partial}{\partial x}
+ \varepsilon^2   \lambda_{6,2}y^2 \frac{\partial}{\partial x}
+ \varepsilon \lambda_{4,1} xy \frac{\partial}{\partial y}$
 \end{itemize}
 represent one irreducible component of the Nash  space of arcs, centred at $\cup_{\lambda \in I_1\cap I_3} E_{\lambda}$.
 Therefore, this is an essential perturbation in the sense of Definition \ref{essentialperturbation}. For brevity, we denote this irreducible component
 $$Arc(I_1\cap I_3) .$$
Note that maximal order of the bifurcation function $M_k$ is   $4$ (contrary to what is affirmed in \cite{BGG}). In \cite{Iliev2}, \cite{Ma}, \cite{Zoladek2}, it is shown that the bifurcation functions are linear combinations of four integrals. The authors compute the bifurcation function up to order three. It is interesting to note that the bifurcation functions of order three (represented by the family of arcs $(\ref{arc1})$) do not cover all possible linear combinations. It is indeed necessary to go to the order four to cover all these combinations (represented by $(\ref{arc2})$ and $(\ref{arc3})$). 

\subsubsection{Perturbations of the  center set $\{I_2\cap I_1\} \cup \{I_2\cap I_3\} \cup \{I_2\cap I_4\} $, see fig.\ref{figuresat} }
In this section we note that the arc spaces corresponding to these sets are in the closure of $$Arc(I_1),Arc(I_2), Arc(I_3), Arc (I_4).$$ Thus, there will be no new essential perturbations in our list.

As $I_2$ is a three-dimensional real plane, then we can represent this singular set in $\R^3= I_2$ as on figure \ref{figuresat}.
Recall that according to our convention we assume that $\lambda^*\not\in I_1\cap I_2 \cap I_3$. 
 The localized Bautin ideal $\mathcal B_{\lambda^*} $ defined by the polynomials  (\ref{kapteynideal}) is  generated by
\begin{itemize}
\item $ \lambda_1, \lambda_5 (\lambda_3-\lambda_6) , \lambda_2 (\lambda_3-\lambda_6)$  if $\lambda^* \in I_2\cap I_1$
\item $\lambda_1, \lambda_5,\lambda_2\lambda_4$  if    $\lambda^* \in I_2 \cap I_3 $
\item $\lambda_1, \lambda_5, \lambda_2 (\lambda_4+5\lambda_3-5 \lambda_6) ,\lambda_2(\lambda_3\lambda_6-2\lambda_6^2-\lambda_2^2)$ if $\lambda^* \in I_2\cap I_4 $ .
\end{itemize}
and we consider instead of (\ref{blowup1}), the maps
\begin{align*}
\lambda \to & [\lambda_1: \lambda_5 (\lambda_3-\lambda_6) :\lambda_2 (\lambda_3-\lambda_6)] \\
\lambda \to & [\lambda_1: \lambda_5:\lambda_2\lambda_4]\\
\lambda \to & [ \lambda_1: \lambda_5:\lambda_2 (\lambda_4+5\lambda_3-5 \lambda_6) : \lambda_2(\lambda_3\lambda_6-2\lambda_6^2-\lambda_2^2)]
\end{align*}

As in the case $I_1\cap I_3$, we verify that the exceptional divisor $E_{\lambda^*}$ is equal to $\mathbb P^2, \mathbb P^2$ and $\mathbb P ^3$ respectively.

Consider first the perturbations   (\ref{arc11}), where $\lambda^* \in I_2\cap I_1$ and
\begin{eqnarray}
\label{arc21}
\begin{aligned}
\lambda_1&= \varepsilon^2  \lambda_{1,2} + O(\varepsilon^3) \\
\lambda_2 &=\varepsilon  \lambda_{2,1} + O(\varepsilon^2) \\
\lambda_5 &= \varepsilon \lambda_{5,1}  + O(\varepsilon^2)\\
\lambda_3-\lambda_6 &=  \varepsilon  (\lambda_{3,1}-\lambda_{6,1})+ O(\varepsilon^2)
\end{aligned}
\end{eqnarray}
with center
$$P=
[\lambda_{1,2}: \lambda_{5,1}(\lambda_{3,1} - \lambda_{6,1}): \lambda_{2,1}(\lambda_{3,1} - \lambda_{6,1}) ] .
$$
A continuous deformation of this family of arcs is
\begin{eqnarray}
\label{arc21def}
\begin{aligned}
\lambda_1&= \varepsilon (\varepsilon + \delta) [\lambda_{1,2} + O(\varepsilon)] \\
\lambda_2 &=(\varepsilon+\delta) [\lambda_{2,1} + O(\varepsilon)]\\
\lambda_5 &=(\varepsilon+\delta) [\lambda_{5,1} + O(\varepsilon)]\\
\lambda_3-\lambda_6 &=  \varepsilon  [(\lambda_{3,1}-\lambda_{6,1})+ O(\varepsilon)]
\end{aligned}
\end{eqnarray}
which shows that the family (\ref{arc21}) is in the closure of $Arc(I_1)$. Thus the exceptional divisor $E_{I_1\cap I_2}$ is ``described" by the closure of $Arc(I_1)$ and there is no new essential deformation here.

The case  $\lambda^* \in I_2\cap I_3$ and
\begin{eqnarray}
\label{arc23}
\begin{aligned}
\lambda_1&= \varepsilon^2  \lambda_{1,2} + O(\varepsilon^3) \\
\lambda_2 &=\varepsilon  \lambda_{2,1} +  O(\varepsilon^2) \\
\lambda_4 &=\varepsilon  \lambda_{4,1} +  O(\varepsilon^2) \\
\lambda_5 &= \varepsilon^2 \lambda_{5,2}  +  O(\varepsilon^3)
\end{aligned}
\end{eqnarray}
with center
$$P=
[\lambda_{1,2}: \lambda_{5,2}: \lambda_{2,1}\lambda_{4,1} ] 
$$
is studied similarly, it belongs to the closure of $Arc(I_2)$ and $Arc(I_3)$. Finally, we consider the case 
$$\lambda^* = (\lambda_{1,0},\lambda_{2,0},\lambda_{3,0},\lambda_{4,0}, \lambda_{5,0}, \lambda_{6,0}) \in I_2\cap I_4$$
and hence
$$
\lambda_{1,0}=\lambda_{2,0}= \lambda_{5,0} = \lambda_{4,0}+5\lambda_{3,0}-5 \lambda_{6,0}
=
\lambda_{6,0} (\lambda_{3,0} -2\lambda_{6,0} )= 0 .
$$
We have to consider therefore two cases  : $\lambda_{6,0}=0$ or
$\lambda_{3,0} -2\lambda_{6,0}= 0$ (corresponding to the two irreducible components of $I_2 \cap I_4$ ). Suppose for instance $\lambda_{6,0}=0$

The family of perturbations
\begin{eqnarray}
\label{arc24}
\begin{aligned}
\lambda_1&= \varepsilon^2  \lambda_{1,2} + O(\varepsilon^3) \\
\lambda_5 &= \varepsilon^2 \lambda_{5,2}  +  O(\varepsilon^3)\\
\lambda_2 &=\varepsilon  \lambda_{2,1} +  O(\varepsilon^2) \\
\lambda_6 &=\varepsilon  \lambda_{6,1} +  O(\varepsilon^2) \\
\lambda_4+5\lambda_3-5 \lambda_6 & = \varepsilon (\lambda_{4,1}+5\lambda_{3,1}-5 \lambda_{6,1} ) +  O(\varepsilon^2)  \\
\lambda_3\lambda_6-2\lambda_6^2-\lambda_2^2 & =
\varepsilon \lambda_{3,0} \lambda_{6,1}+ O(\varepsilon^2)
\end{aligned}
\end{eqnarray}
with center
$$P=
[\lambda_{1,2}: \lambda_{5,2}: \lambda_{2,1}(\lambda_{4,1}+5\lambda_{3,1}-5 \lambda_{6,1} ):   \lambda_{2,1}  \lambda_{3,0} \lambda_{6,1} ] .
$$
describes the exceptional divisor $E_{\lambda^*} = \mathbb P^3$.

Consider a small deformation in $\delta$ of the initial family (\ref{arc24}) which is of the form
\begin{eqnarray}
\label{arc24def}
\begin{aligned}
\lambda_2 &= (\varepsilon + \delta)[\lambda_{2,1} +  O(\varepsilon)] \\
\lambda_3\lambda_6-2\lambda_6^2-\lambda_2^2 & =\varepsilon[ \lambda_{3,1} \lambda_{6,1}-2\lambda_{6,1}^2-\lambda_{2,1}^2 + 0(\delta)]+ O(\varepsilon^2)\\
\lambda_4+5\lambda_3-5 \lambda_6 & = \varepsilon (\lambda_{4,1}+5\lambda_{3,1}-5 \lambda_{6,1} ) +  O(\varepsilon^2)  \\
\lambda_1&= \varepsilon (\varepsilon + \delta) [ \lambda_{1,2} + O(\varepsilon)] \\
\lambda_5 &= \varepsilon (\varepsilon + \delta) [ \lambda_{5,2} + O(\varepsilon)] 
\end{aligned}
\end{eqnarray}
The family (\ref{arc24def}) induces a continuous deformation of (\ref{arc11}) which shows that  (\ref{arc24}) is in the closure of  (\ref{arc24def}). The case 
$\lambda_{3,0} -2\lambda_{6,0}= 0$, $\lambda_{6,0}\neq 0$ is analogous.
We conclude that  $E_{I_2\cap I_4}$ is described by (the closure of)  $Arc(I_4)$, and there is no new essential perturbation here.
\subsubsection{  Perturbations of the  center set $\{I_1\cap I_2 \cap I_3\}$ }
The center set  
$$\{I_1\cap I_2 \cap I_3\} = \{ \lambda_1= \lambda_2= \lambda_3-\lambda_6 = \lambda_4=\lambda_5= 0 \}
$$
 is a straight line and we assume that $\lambda_3\neq 0$. The
 localized Bautin ideal $\mathcal B_{\lambda^*} $ defined by the polynomials  (\ref{kapteynideal}) is also generated by  
\begin{equation}
\lambda_1,
\lambda_5(\lambda_3-\lambda_6),
\lambda_2\lambda_4(\lambda_3-\lambda_6)(\lambda_4+5\lambda_3-5\lambda_6),
\lambda_2\lambda_4(\lambda_3-\lambda_6)^2
\end{equation}
or equivalently
\begin{equation}
\label{v123}
\lambda_1,
\lambda_5(\lambda_3-\lambda_6),
\lambda_2\lambda_4^2(\lambda_3-\lambda_6),
\lambda_2\lambda_4(\lambda_3-\lambda_6)^2
\end{equation}
Therefore, instead of (\ref{blowup1}), we use the map
\begin{align}
\lambda \to [\lambda_1: \lambda_5(\lambda_3-\lambda_6):
\lambda_2\lambda_4^2(\lambda_3-\lambda_6):
\lambda_2\lambda_4(\lambda_3-\lambda_6)^2 ] .
\end{align}

The family  (\ref{arc11}), where $\lambda^* \in I_1\cap I_2\cap I_3$ and induced by
\begin{eqnarray}
\label{arc123}
\begin{aligned}
\lambda_1&= \varepsilon^4  \lambda_{1,4} + O(\varepsilon^5) \\
\lambda_5 &= \varepsilon^3 \lambda_{5,3}  +  O(\varepsilon^4)\\
\lambda_2 &=\varepsilon  \lambda_{2,1} +  O(\varepsilon^2) \\
\lambda_4 &=\varepsilon  \lambda_{4,1} +  O(\varepsilon^2) \\
\lambda_3-\lambda_6 &=\varepsilon ( \lambda_{3,1} -\lambda_{6,1})+  O(\varepsilon^2) 
\end{aligned}
\end{eqnarray}
with center
$$P=
[\lambda_{1,4}: \lambda_{5,3}(\lambda_{3,1}- \lambda_{6,1} ):   \lambda_{2,1} \lambda_{4,1}^2(\lambda_{3,1}- \lambda_{6,1} ):
 \lambda_{2,1} \lambda_{4,1}(\lambda_{3,1}- \lambda_{6,1} )^2] 
$$
describes an open dense subset of $E_{\lambda^*}= \mathbb P^3$. 

There is a continuous deformation of (\ref{arc11}) induced by the deformation 
\begin{eqnarray}
\label{arc123def}
\begin{aligned}
\lambda_1&= \varepsilon^3(\varepsilon + \delta)[\lambda_{1,4} + O(\varepsilon)]\\
\lambda_5 &= \varepsilon^2(\varepsilon+\delta) [\lambda_{5,3}  + O(\varepsilon)]\\
\lambda_2 &=(\varepsilon +\delta)[\lambda_{2,1} + O(\varepsilon)] \\
\lambda_4 &=\varepsilon  \lambda_{4,1} + O(\varepsilon^2) \\
\lambda_3-\lambda_6 &=\varepsilon ( \lambda_{3,1} -\lambda_{6,1})+  O(\varepsilon^2) 
\end{aligned}
\end{eqnarray}
which shows that this family (\ref{arc123}) belongs to the closure of $Arc(I_1\cap I_3)$, and again there is no new essential perturbation.
\subsubsection{Perturbations of the linear center $\lambda^*=0$}
Consider finally the singular point $\lambda^* = (0,\dots,0)$ on the center set $Z(\mathbf B)$, which corresponds to the linear center
$$
X_0= - y  \frac{\partial}{\partial x} + x  \frac{\partial}{\partial y} .
$$
This point is the intersection of the four centre sets  $I_1,I_2,I_3$ and $I_4$, and we shall show that  
 $Arc(I_1\cap I_2\cap I_3 \cap I_4)$ is in the closure of $Arc(I_1\cap I_2\cap I_3)$.

Let us recall, that the localized Bautin ideal at the origin is generated by the polynomials
\begin{align*}
v_1(\lambda)&=\lambda_1,\\
v_2(\lambda)&=\lambda_5(\lambda_3-\lambda_6),\\
v_3(\lambda)&=\lambda_2\lambda_4(\lambda_3-\lambda_6)(\lambda_4+5\lambda_3-5\lambda_6),\\
v_4(\lambda)&=\lambda_2\lambda_4(\lambda_3-\lambda_6)^2(\lambda_3\lambda_6-2\lambda_6^2-\lambda_2^2)
\end{align*}

The family  (\ref{arc11}) 
\begin{align*}
\varepsilon \to \left( \lambda(\varepsilon), [v_1(\lambda(\varepsilon)):v_2(\lambda(\varepsilon)): \dots : v_6(\lambda(\varepsilon)) ] \right) 
\end{align*}
induced by 
\begin{eqnarray}
\label{arc000}
\begin{aligned}
\lambda_1&= \varepsilon^6  \lambda_{1,6} +O(\varepsilon^7)\\
\lambda_2 &=\varepsilon  \lambda_{2,1} + O(\varepsilon^2) \\
\lambda_3 &=\varepsilon  \lambda_{3,1} +O(\varepsilon^2) \\
\lambda_4 &=\varepsilon  \lambda_{4,1} + O(\varepsilon^2) \\
\lambda_5 &= \varepsilon^5 \lambda_{5,5}  +O(\varepsilon^6)\\
\lambda_6 &=\varepsilon  \lambda_{6,1} + O(\varepsilon^2) \\
\lambda_4+5\lambda_3-5 \lambda_6 & = \varepsilon^3 (\lambda_{4,3}+5\lambda_{3,3}-5 \lambda_{6,3} ) + O(\varepsilon^4) 
\end{aligned}
\end{eqnarray}
with center
$$P=
[\lambda_{1,6}: \lambda_{5,5}(\lambda_{3,1}- \lambda_{6,1} ):   \lambda_{2,1} \lambda_{4,1}(\lambda_{3,1}- \lambda_{6,1} )(\lambda_{4,3} +5\lambda_{3,3}-5\lambda_{6,3})
$$
$$
: \lambda_{2,1} \lambda_{4,1}(\lambda_{3,1}- \lambda_{6,1} )^2 ( \lambda_{3,1} \lambda_{6,1}-2\lambda_{6,1}^2-\lambda_{2,1}^2)] 
$$
describes an open subset of the exceptional divisor $E_{\lambda^*} = \mathbb P^3$. The closure of the set of centres $P$ equals   $E_{\lambda^*}$. 

We shall show that a general arc ( induced by ) (\ref{arc000}) has a suitable deformation  $ (\ref{arc000}^{\delta})$, continuous in the topology of the Nash arc space, which is  of the form (\ref{arc123}).
It is moreover a continuous deformation in the sense of the arc space topology.

We define first $\lambda_3^\delta =\lambda_3(\varepsilon)+\delta, \lambda_6^\delta = \lambda_6(\varepsilon)+\delta$, and $\lambda_2^\delta = \lambda_2$.
As $\lambda_3\lambda_6-2\lambda_6^2-\lambda_2^2$ as a power series in $\varepsilon$ has a double zero at $\varepsilon=0$, then
 we obtain
\begin{align*}
 \lambda_3^\delta \lambda_6^\delta- 2 (\lambda_6^\delta)^2-\lambda_2^2 
& = 
(\lambda_3+\delta)(\lambda_6+\delta)-2(\lambda_6+\delta)^2-\lambda_2^2 &  \\
&= 
[\varepsilon^2 + \varepsilon_1(\delta) \varepsilon + \varepsilon_2(\delta) ] 
[p_2(\delta) + O(\varepsilon)]
\end{align*}
where $\varepsilon_1(\delta), \varepsilon_2(\delta), p_2(\delta)$ are analytic functions in $\delta$, and 
$$\varepsilon_1(0)=\varepsilon_2(0)=0,\; p_2(0)= \lambda_{3,1} \lambda_{6,1}-2\lambda_{6,1}^2-\lambda_{2,1}^2. $$
We  define $\lambda_{4}^\delta$ in such a way, that
\begin{align*}
\lambda_4^\delta+5\lambda_3-5 \lambda_6 & = \varepsilon [\varepsilon^2 + \varepsilon_1(\delta) \varepsilon + \varepsilon_2(\delta) ]  [p_1(\delta)  + O(\varepsilon)] 
\end{align*}
where $p_1(\delta)$ is analytic in $\delta$ and
$$
p_1(0)= \lambda_{4,3}+5\lambda_{3,3}-5 \lambda_{6,3}  .
$$
Finally, the power series $X^\delta, \lambda_5^\delta$ are defined similarly by the conditions
\begin{align*}
X^\delta & = \varepsilon^4 [\varepsilon^2 + \varepsilon_1(\delta) \varepsilon + \varepsilon_2(\delta) ] [\lambda_{1,6}^\delta +O(\varepsilon)] \\
\lambda_5^\delta & =  \varepsilon^3 [\varepsilon^2 + \varepsilon_1(\delta) \varepsilon + \varepsilon_2(\delta) ] [\lambda_{5,5}^\delta +0(\varepsilon)] 
\end{align*}
where $\lambda_{1,6}^\delta, \lambda_{5,5}^\delta$ depend analytically in $\delta$ and  $\lambda_{1,6}^0= \lambda_{1,6}, \lambda_{5,5}^0= \lambda_{5,5} $.
The $\delta$-family of arcs $ (\ref{arc000}^{\delta})$ induced by the power series $\delta \to \lambda_i^\delta$ has a center
$$P^\delta=
[\lambda_{1,6}^\delta: \lambda_{5,5}^\delta(\lambda_{3,1}- \lambda_{6,1} ):   \lambda_{2,1} \lambda_{4,1}^\delta(\lambda_{3,1}- \lambda_{6,1} )p_1(\delta)
$$
$$
: \lambda_{2,1} \lambda_{4,1}^\delta(\lambda_{3,1}- \lambda_{6,1} )^2 p_2(\delta)]  .
$$

 This completes the proof that $Arc(I_1\cap I_2\cap I_3 \cap I_4)$ is in the closure of $Arc(I_1\cap I_2\cap I_3)$, so there is no new essential perturbation again.

To the end of this section we discuss the bifurcation functions in the quadratically perturbed linear center in the context of the inclusion
\begin{equation}
\label{inclusion}
Arc(I_1\cap I_2\cap I_3 \cap I_4)\subset Arc(I_1\cap I_2\cap I_3) .
\end{equation}
The set $I_1\cap I_2\cap I_3 \cap I_4$ is just one point (the linear center), $I_1\cap I_2\cap I_3$ is a two-plane representing "Hamiltonian triangles", that is to say, Hamiltonian systems in which the Hamiltonian is a product of three linear factor. The inclusion (\ref{inclusion}) means that a bifurcation function of the perturbed linear centre is either a limit, or a limit of derivatives of bifurcation functions, related to the Hamiltonian triangle case. Recall, that in the Hamiltonian triangle case, we have three bifurcation functions which are complete elliptic integrals of first, second and third kind, and the fourth one is an iterated integral of length two \cite{Iliev}. After "taking the limit" the Hamiltonian takes the form $h=x^2+y^2$ and the genus of integrals drop to zero. As we shall see, they become polynomials of degree at most four in $h$, vanishing at the origin. This is the content of the classical theorem of Zoladek \cite[Theorem 4] {Zoladek1} which we recall now. 
Denote by $\mathcal P $ the Poincaré return map associated to the perturbed linear center, parameterized by the Hamiltonian $h=\frac{x^2+y^2}{2}$. Then
$$
\mathcal P (h) - h = 2\pi v_1 h (1+O(\lambda)) + v_2 h^2 (1+O(\lambda)) + v_3 h^3 (1+O(\lambda)) + v_4 h^4 (1+O(\lambda)) .
$$
By abuse of notations here $v_2,v_3,v_4$ are the polynomials above, but up to multiplication by a non zero constant.

$ O(\lambda)$ means a convergent power series in $h$ whose coefficients are analytic in $v_1$, polynomial in $v_2,v_3,v_4$ and belong to the ideal generated by $v_1,v_2,v_3,v_4$ in
$\R [v_2,v_3,v_4]\{v_1\}$. This  last property is crucial for the computation of the bifurcation functions. We conclude that every bifurcation function is a polynomial of the form
$c_1 h+ c_2 h^2 + c_3 h^3 + c_4 h^4$.
\vskip 1pt
In \cite{Iliev1}, Iliev extended Zoladek's theorem to perturbations of the harmonic oscillator of any degree using the algorithm of \cite{F}. A complete presentation of this result is reproduced in the book \cite{P}. In this book, on page 474, the author writes: 
\begin{quotation}
we believe that every row in table 1 will stabilize at some value $N(n)$ for all $k\geq K(n)$
\end{quotation}
This is indeed a consequence of the 
Theorem $5$ of our article. It applies as well to the perturbation of the Bogdanov-Takens Hamiltonian and the table 4.2, page 477.

\section{Conclusion and perspectives}

To conclude, we resume the main new points of our approach and discuss further possible developments, for instance, to the local Hilbert's 16th problem on a period annulus for polynomial perturbations of any degree.\\

1-In this article, we have represented the set of bifurcation functions (Melnikov functions of any order) by the exceptional divisor $E$ of the canonical blow-up of the Bautin ideal (cf. Proposition $3$) and define the corresponding Iliev number $k$ , Definition \ref{def1}. In the particular case of the Kapteyn normal form of quadratic deformations, we have checked that this set is always a vector space (or equivalently that $E_\lambda$ is always a projective space). Is it true in general for a polynomial perturbation of any degree?\\

2-Our setting allows a quick and systematic computation of the maximal order of the bifurcation function. It does not provide of course {\it a priori} information on the number of zeros of this bifurcation function. Many other techniques have been developped for solving this final step. Finding an explicit integral and an integrating factor for the perturbed center allows to represent the bifurcation function as an (iterated) integral over the level set of the first integral (cf. \cite{F}, \cite{Gavrilov2}). In the known cases, this bifurcation function is a solution of a differential system and techniques like Chebycheff systems, argument principle can be used. For instance, this number of limit cycles is $2$ for generic perturbations of a Hamiltonian in $I_3$ (\cite{Gavrilov3}). It is $2$ in the perturbation of a generic Lotka-Volterra system in $I_1$ and $3$ for the Hamiltonian triangle in $I_1\cap I_3$ (\cite{Iliev2, Ma, Zoladek2}), it is less than $5$ in the case of generic perturbation of $I_4$ (\cite{GI2, Z}). The case of the perturbation of a generic reversible center in $I_2$ is still open. By our setting, we know that computation of the first order bifurcation function is enough in that case.\\

In view of our computations, we conjecture that if there is  a uniform bound $N$ for the number of zeros of the bifurcation functions associated to a family of arcs $A$, then the number of zeros of the bifurcation functions associated to arcs which belong to the closure $\overline{A}$ of $A$ is also less than $N$.  \\

For the Hamiltonian non-generic cases, the intersection $I_2\cap I_3$ has been fully covered by many contributions including \cite{Iliev4,LL3}, \cite{HI, CLL}, as well as the intersection $I_1\cap I_3$, \cite{ZLL, ZZ} and the cyclicity is $2$, except for the Hamiltonian triangle \cite{Iliev2}.\\

For the generic cases of $I_1\cap I_2$, the bound is $2$ (cf. \cite{LL2}).\\

Henryk Zoladek conjectured in (\cite{Zoladek1}, p. 244) that
\begin{quotation}
The maximal number of limit cycles appearing after perturbation of the system with center as a function of a point of the center manifold is equal to the maximum of the values of the number of zeroes of bifurcation function in a neighborhood of the point in the center manifold.
\end{quotation}
 This conjecture can be reformulated in restriction to each components of the Nash space of arcs $ Arc(B_I\C^n,E)$.\\

3-Our setting is well adapted to discuss the confluence phenomenon that we mention in the beginning of paragraph $4$. To explain it with more details, let us consider first a smooth point $p$ of the stratum $I_2$ of the centre manifold. Corresponding to the center $p$, there is an associated logarithmic integral $H$ and an integrating factor $M$. It is enough to consider a bifurcation function of order $1$ and it can be represented as a Melnikov-Pontryagin integral over the closed level sets $H=h$. For the essential perturbation defined in $4.1$, this bifurcation function can be written as:

\begin{equation}
\label{I2}
M_1(h)=\lambda_1J_1(h)+\lambda_2 J_2(h)+\lambda_5 J_5(h).
\end{equation}

The set of generic reversible centers can be parametrized by (after a scaling of coordinates, assume $B=2$, then $A=a\in\mathbb{R}, C=b\in\mathbb{R}$):

\begin{equation}
\label{parI_2}
\dot{z}=-{\rm i}z+az^2+2\mid z\mid^2+b\overline{z}^2.
\end{equation}

Note that $LV\cap R$ or $I_1\cap I_2$, cannot be described in this chart. The intersection $I_2\cap I_3$ is given by $a=-1$ and the intersection $I_2\cap I_4$ by $a=4, b=\pm 2$.

Consider now, generic centers on $I_2\cap I_3$. We have checked that the family of arcs (\ref{arc23}) can be used. This means that bifurcation function of second order are enough and that the associated arcs can be described as limits of arcs for $I_2$ (or $I_3$). Explicit computations of Iliev show that when $a=-1$, the integral $I_5$ vanishes and that the second order bifurcation function can be choosen as:

\begin{equation}
\label{I2I3}
M_2(h)=\lambda_1J_1(h)+\lambda_2 J_2(h)+\lambda_5 \frac{d}{da}J_5(h)\mid_{a=-1}.
\end{equation}

This is an example of what could be called a confluence phenomenon. Consider next, centers on $I_2\cap I_4$. In that case, we have shown that the arcs (\ref{arc24}) can be used, and in particular that second-order bifurcation functions are enough. But we have also shown that the family of arcs (\ref{arc24}) can be represented as a limit of the family of arcs (\ref{arc24def}), which are of the type associated with $I_4$. The explicit computation made by Iliev in that case matches the deformation of arcs and yields:

\begin{equation}
\label{I2I4}
M_2(h)=\lambda_1J_1(h)+\lambda_2 J_2(h)+[\lambda_4 \frac{d}{db}J_5(h)+\lambda_5 \frac{d}{da}J_5(h)]\mid_{a=4, b=\pm 2}.
\end{equation}

Similar computations can be made in the case $I_1\cap I_2$. A center which is in $LV$ corresponds to $B=0$, $A=1$ (after a scaling) and $C=b+{\rm i}c$:

\begin{equation}
\label{parI_1}
\dot{z}=-{\rm i}z+z^2+(b+{\rm i}c)\overline{z}^2.
\end{equation}

Consider a smooth point $p$ of the stratum $I_1$ of the centre manifold. Corresponding to the center $p$, there is an associated logarithmic integral $H$ and an integrating factor $M$. It is enough to consider a bifurcation function of order $1$ and it can be represented as a Melnikov-Pontryagin integral over the closed level sets $H=h$. For the essential perturbation defined in $4.1$, this bifurcation function can be written as:

\begin{equation}
\label{I_1}
M_1(h)=\lambda_1J_1(h)+\lambda_2 J_2(h)+\lambda_3 J_3(h).
\end{equation}

The generic centers of $I_1\cap I_2$ corresponds to $c=0$. Iliev showed that $J_2(h)$ vanishes on $c=0$. The deformation of arcs (\ref{arc21}) into (\ref{arc21def}) corresponds to the confluence observed by Iliev  \cite[Corollary 1]{Iliev} :

\begin{equation}
\label{I_2}
M_2(h)=\lambda_1J_1(h)+\lambda_2 \frac{d}{dc} J_2(h)\mid_{c=0}+\lambda_3 J_3(h).
\end{equation}\\

4- We expect many further developements, for instance for Abel equations \cite{BFY}, or quadratic double centers \cite{Andronova1,Andronova2,Gentes1,Gentes2,Iliev5, LL}, where the Bautin ideal is explicitly known.

\appendix
\section{Logarithmic First Integral and integrating Factor}

For completeness, we reproduce bellow  the original result of Dulac, classifying centers of quadratic plane vector fields. We give then its modern geometric counterpart  - Theorem \ref{dulac}.
Namely, consider the 12-dimensional vector space $Q$ of polynomial one-forms $$\omega = P(x,y) dx + Q(x,y) dy$$ where $P,Q$ are polynomials of degree two.
Each $\omega$ defines a quadratic vector field $$X = Q\frac{\partial}{\partial x} - P \frac{\partial}{\partial y} .$$ Suppose that $X$ (or $\omega$) is real and has a a center. In this case, near the center critical point in $\R^2$ we have an analytic first integral having a Morse critical point. More generally, we say that a complex analytic plane vector field $X$ (or $\omega$) has a Morse critical point, provided that in a neighbourhood of some singular point it has an analytic first integral with Morse critical point. This notion generalizes the notion of a center, and has a meaning for vector fields with complex coefficients. The Dulac's Theorem classifies complex quadratic vector fields having a Morse critical point. A modern account of this  is given in Cerveau and Lins Neto \cite{cerlin96}, and we reproduce it here\\
{\bf Theorem}(Dulac \cite{Dulac})
\emph{
Let $X$ be a complex quadratic vector field with associated one-form $\omega$.  
$X$ has a Morse critical point, if and only if $\omega$ falls in one of the following 12 cases 
\begin{description}
\item[(a) ] 
$\omega= d q, \deg q = 3$
\item[(b)] 
$
\omega= p_1p_2p_3 \cdot\eta,\;  \eta = \lambda_1\frac{dp_1}{p_1}+  \lambda_2 \frac{dp_2}{p_2}+  \lambda_3 \frac{dp_3}{p_3}, \deg p_1 = \deg p_2=\deg p_3=1
$
\item[(c)]
$\omega= p_1p_2 \cdot \eta,\;  \eta = \lambda_1\frac{dp_1}{p_1}+  \lambda_2 \frac{dp_2}{p_2}, \deg p_1 = 2, \deg p_2 = 1$
\item[(d)]
$\omega= p_1p_2 \cdot \eta,\;  \eta = \lambda_1\frac{dp_1}{p_1}+  \lambda_2 \frac{dp_2}{p_2} + d q, \deg p_1 = \deg p_2 = \deg q =1$
\item[(e)]
$\omega= p_1p_2 \cdot \eta,\;  \eta = \lambda_1\frac{dp_1}{p_1}+  \lambda_2 \frac{dp_2}{p_2} + d \frac{q}{p_1}, \deg p_1 = \deg p_2 = \deg q =1$
\item[(f)]
$\omega= p^3 \cdot \eta,\;  \eta = \frac{dp}{p}+ d \frac{q}{p^2},   \deg p = 1, \deg q =2$
\item[(g)]
$\omega= p^2 \cdot \eta,\;  \eta = \frac{dp}{p}+ d \frac{q}{p},   \deg p = 1, \deg q =2$
\item[(h)]
$\omega= p \cdot \eta,\;  \eta = \frac{dp}{p}+ d q,   \deg p = 1, \deg q =2$
\item[(i)]
$\omega= p \cdot \eta,\;  \eta = \frac{dp}{p}+ d q,   \deg p = 2, \deg q =1$
\item[(j)]
$\omega= f g  \cdot \eta,\;  \eta = 3 \frac{df}{f} - 2  \frac{dg}{g},   \deg f = 2, \deg g =3$ .
\end{description}
}
In the first three cases (a), (b), (c), and in the last one (j) the one-form $\omega$ can be written as
\begin{equation}
\label{log}
\omega=f_1...f_s(\Sigma_{i=1}^s \lambda_i\frac{df_i}{f_i}),
\end{equation}
where $f_i$ are  polynomials with suitable complex coefficients. The first integral $f$ of (\ref{log}) is of logarithmic type
$f = f_1^{\lambda_1}...f_s^{\lambda_s}$. 
Following
Movasati \cite{Movasati},  for given positive integers $d_1,...,d_s$, we denote by $\mathcal{L}(d_1,...,d_s)$ the set of polynomial one-forms $\omega_0 $ (\ref{log}), where  $f_i$ are complex polynomials of degree  $d_i$ , $\lambda_i \in \C$,
$i=1,...,s$. The algebraic closure $\overline{\mathcal{L}(d_1,...,d_s)}$ of  $\mathcal{L}(d_1,...,d_s)$ is then an irreducible algebraic subset of the vector space of polynomial one-forms of degree at most $d=\Sigma_{i=1}^s d_i-1$. \emph{It is a remarkable fact, that one-forms of type (d),(e),(f),(g),(h),(i) above are limits of one-forms from the sets (a), (b), (c) (j).} 
This leads to the following simpler formulation of the Dulac's  theorem, which is   implicit in Zoladek \cite[Theorem 1]{Zoladek1},  and explicit in Lins Neto \cite[Theorem 1.1]{lins14}.
\begin{theorem}
\label{dulac}
Let $Q^\C$ be the 12-dimensional complex vector space of quadratic plane differential systems. The algebraic closure of the subset of quadratic systems having 
a Morse critical point is an algebraic subset of $Q^\C$  with irreducible decomposition as follows
$$
\overline{\mathcal{L}(3)}, \overline{\mathcal{L}(2,1)}, \overline{\mathcal{L}(1,1,1)}, \overline{ \mathcal{L}(3,2)}\cap Q^\C .
$$
\end{theorem}

 The usual terminology for these four components in the real case  is, according to (\ref{strata}) : Hamiltonian $H = \overline{\mathcal{L}(3)} \cap Q^\R$ , reversible $R= \overline{\mathcal{L}(2,1)} \cap Q^\R$, Lotka-Volterra $LV = \overline{\mathcal{L}(1,1,1)} \cap Q^\R$ and co-dimension four $Q_4= \overline{ \mathcal{L}(3,2)}\cap Q^\R$ component of the center set, respectively.
Another terminology is introduced in  \cite[section 13]{ilya08}. \vskip 1pt

Some more explanation should be given about $Q_4$. In that case, associated with $\mid B\mid=\mid C\mid$ there is a parameter $\alpha={\rm cos}(\xi/2)$ so that:

\begin{eqnarray}
\label{q4}
\begin{aligned}
f_2=x^2+4y+1 \\
f_3={\alpha}x(x^2+6y)+6y+1,
\end{aligned}
\end{eqnarray}
and it can be checked that the form $\omega_0=3f_3df_2-2f_2df_3$ is of degree $2$.

\begin{proof}[Sketch of the proof of Theorem \ref{dulac}]
Let $X$ be a quadratic differential system with associated one-form $\omega$.  
The cases (a), (b), (c), (j) correspond obviously to $\omega$ in 
$$\mathcal{L}(3), \mathcal{L}(1,1,1), \mathcal{L}(2,1), \mathcal{L}(3,2)\cap Q^\C .$$
When the parameter $\varepsilon$ tends to $0$, the one-form
$$
\omega_\varepsilon = p_1p_2 (1+ \varepsilon q)( \lambda_1\frac{dp_1}{p_1}+  \lambda_2 \frac{dp_2}{p_2} +\frac{1}{\varepsilon} \frac {d(1+ \varepsilon q)}{1+ \varepsilon q}) \in \mathcal{L}(1,1,1) 
$$
tends to
$$
\omega_0= p_1p_2 \cdot \eta_0,\;  \eta_0 = \lambda_1\frac{dp_1}{p_1}+  \lambda_2 \frac{dp_2}{p_2} + d q
$$
which shows that in the case (d) the one-forms $\omega$ belong to $\overline{\mathcal{L}(1,1,1)}$. Similarly, the one-form
$
\omega_\varepsilon= p_1p_2 (p_1 + \varepsilon q)  \cdot \eta_\varepsilon \in \mathcal{L}(1,1,1)
$
where
$$
 \eta_\varepsilon = \lambda_1\frac{dp_1}{p_1}+  \lambda_2 \frac{dp_2}{p_2} + \frac1\varepsilon (\frac{d(p_1+\varepsilon q)}{ p_1+\varepsilon q} - \frac{dp_1}{p_1} )
$$
tends to the form 
$$\omega_0= p_1p_2 \cdot \eta_0,\;  \eta_0 = \lambda_1\frac{dp_1}{p_1}+  \lambda_2 \frac{dp_2}{p_2} + d \frac{q}{p_1}, \deg p_1 = \deg p_2 = \deg q =1.$$
This shows that in the case $(e)$ the one forms $\omega$ belong to $\overline{\mathcal{L}(1,1,1)}$. The remaining cases (f)-(i) are treated in a similar way, and they all belong to $\overline{\mathcal{L}(1,2)}$.

Finally, the irreducibility of the algebraic sets $\overline{\mathcal{L}(3)}, \overline{\mathcal{L}(2,1)}, \overline{\mathcal{L}(1,1,1)}$ follows from the fact that they are naturally parameterised by the coefficients of the polynomials $p_i$ and the exponents $\lambda_j$. The irreducibility of $\overline{ \mathcal{L}(3,2)}\cap Q^\C$ follows from the parameterisation (\ref{q4}).
\end{proof}

  \begin{acknowledgements}
  The authors thank the referee for the valuable remarks. JPF was
supported by the SJTU-INS Research Project for Visiting Scholar.
LG  has been partially supported by the Grant No DN 02-5 of the Bulgarian Fund “Scientific Research".
DX is supported by the NSFC grants (No. 11431008  $\&$No.11371248).

  \end{acknowledgements}

\end{document}